\documentclass[11pt]{article}
\usepackage{amsfonts,amsmath,amssymb}
\usepackage{enumerate, graphicx}
\usepackage{xspace}
\usepackage{boxedminipage}
\usepackage{url}
\usepackage{algo}
\usepackage{enumerate, paralist}
 \usepackage[usenames,dvipsnames]{pstricks}
 \usepackage{pst-grad} 
 \usepackage{pst-plot} 
 \usepackage{bbm} 
 \usepackage{float} 
\usepackage{hyperref} 
\usepackage{multirow} 
\usepackage{algorithm} 
\usepackage[bb=boondox]{mathalfa} 
\usepackage{amsmath,amssymb,amsthm}
\usepackage{mathtools} 
\usepackage{thmtools, thm-restate}
\usepackage{soul} 
\usepackage{} 
\usepackage{todonotes}
\usepackage{diagbox}

\usepackage{algorithm} 
\usepackage[noend]{algpseudocode} 

\hypersetup{
    colorlinks,
    linkcolor={Red},
    citecolor={ForestGreen},
    urlcolor={Violet}
}

\usepackage[margin=1in]{geometry}
\usepackage{setspace}
\usepackage[english]{babel}
\usepackage[utf8]{inputenc}
\usepackage{graphicx}

\usepackage{boxedminipage}
\usepackage{alltt}

\newtheorem*{theorem*}{Theorem}

\newtheorem{lemma}{Lemma}[section]
\newtheorem{theorem}[lemma]{Theorem}

\newtheorem{proposition}{Proposition}
\newtheorem{defn}[lemma]{Definition}

\theoremstyle{definition}

\newcommand{\mymath}[1]{\newline \centerline{$\displaystyle{#1}$}}

\def\floor#1{\lfloor {#1} \rfloor}
\def\ceil#1{\lceil {#1} \rceil}

\newcommand\defeq{\mathrel{\overset{\makebox[0pt]{\mbox{\normalfont\tiny\sffamily def}}}{=}}}

\def\ind{\mathbbm{1}}

\newcommand\Var[2][]{
\ensuremath{\mathrm{\mathbf{Var}}_{#1} 
\pmb{\left[\vphantom{#2}\right.}
{#2}
\pmb{\left.\vphantom{#2}\right]}
}}

\newcommand\E[2][]{
\ensuremath{\mathrm{\mathbf{E}}_{#1} 
\pmb{\left[\vphantom{#2}\right.}
{#2}
\pmb{\left.\vphantom{#2}\right]}
}}

\renewcommand\Pr[2][]{
\ensuremath{\mathrm{\mathbf{Pr}}_{#1} 
\pmb{\left[\vphantom{#2}\right.}
{#2}
\pmb{\left.\vphantom{#2}\right]}
}}

\def\Poi{\ensuremath{\mathrm{Poi}}}

\def\accept{{\fontfamily{cmss}\selectfont accept}\xspace}
\def\reject{{\fontfamily{cmss}\selectfont reject}\xspace}

\def\BB{\mathcal{B}}
\def\CC{\mathcal{C}}
\def\DD{\mathcal{D}}
\def\FF{\mathcal{F}}
\def\II{\mathcal{I}}

\def\MM{\mathcal{M}}

\def\PP{\Pi}

\def\UU{\mathcal{U}}

\def\vv{v}

\def\poi{\ensuremath{\mbox{\sl Poi}}}

\newif\ifobviousProofset
\obviousProofsettrue

\newcommand{\obviousProof}[1]{}
\ifobviousProofset
\renewcommand{\obviousProof}[1]{
  \newline
  \textcolor{blue}{
  	\textbf {Proof for our own record:} #1
  	{\hspace*{\fill}$\Box$}
  	\newline
  }
}
\fi

\makeatletter
\newcommand*{\rom}[1]{\expandafter\@slowromancap\romannumeral #1@}
\makeatother

\newcommand{\inCOLT}[2]{#2}

\begin{document}
\title{Testing Mixtures of Discrete Distributions}

\author{
Maryam Aliakbarpour\thanks{MA is supported by funds from the MIT-IBM Watson AI Lab (Agreement No. W1771646),  the NSF grants IIS-1741137, and CCF-1733808.}\\
CSAIL, MIT\\
\texttt{maryama@mit.edu} \\
\and
Ravi Kumar\\
Google \\
\texttt{ravi.k53@gmail.com}
\and
Ronitt Rubinfeld \thanks{RR is supported 
by funds from the MIT-IBM Watson AI Lab (Agreement No. W1771646)
the NSF grants CCF-1650733, CCF-1733808, IIS-1741137 and CCF-1740751.}\\
CSAIL, MIT, TAU\\
\texttt{ronitt@csail.mit.edu} \\
}

\date{}
\maketitle

\begin{abstract}
There has been significant study on the sample complexity of testing properties of distributions over large domains.  For many properties, it is known that the sample complexity can be substantially smaller than the domain size. For example, over a domain of size $n$, distinguishing the uniform distribution from distributions that are far from uniform in $\ell_1$-distance uses only $O(\sqrt{n})$ samples. 

However, the picture is very different in the presence of arbitrary noise, even when the amount of noise is quite small.  In this case, one must distinguish if samples are coming from a distribution
that is $\epsilon$-close to uniform from the case where the distribution is $(1-\epsilon)$-far from uniform.  The latter task requires nearly linear in $n$ samples~\cite{Valiant08,VV11}.

In this work, we present a noise model that on one hand is more tractable for the testing problem, and on the other hand represents a rich class of noise families.   In our model, the noisy distribution is a mixture of the original distribution and noise, where the latter is known to the tester either explicitly or via sample access; the form of the noise is also known \emph{a priori}.  Focusing on the identity and closeness testing problems leads to the following mixture testing question:  Given samples of distributions $p, q_1,q_2$, can we test if $p$ is a mixture of $q_1$ and $q_2$?  We consider this general question in various scenarios that differ in terms of how the tester can access the distributions, and show that indeed this problem is more tractable. 
Our results  show that the sample complexity of our testers are exactly the same as for the classical non-mixture case. 
\end{abstract}

\newcommand{\coarse}[1]{\langle #1 \rangle}
\newcommand{\calB}{\mathcal{B}}
\newcommand{\tuple}[1]{\langle #1 \rangle}

\section{Introduction}

Distribution testing~\cite{BFRSW} has been studied extensively for the past many years~(see~\cite{canonne2015survey} for a survey).  In the vanilla version, the problem is to quickly test if a discrete distribution has a certain property or is statistically far from any distribution with that property.  The tester has access to samples from the distribution and strives to be as frugal as possible in the number of samples it uses.  Many statistical properties, including various distances between distributions, are well understood in this model.  There have been several relaxations to the basic testing model including tolerant testing (where the tester should also accept if the distribution is close to having a property), the conditional samples model (where the tester can access the distribution conditioned on a specified subset), making stylized assumptions about the distribution (monotone, sparse support, high-dimensional, etc), and so on.  In each of these works, the aim has been to push the boundaries of our understanding: when do sample-efficient testers exist?  Here, by sample-efficient, we mean the number of samples should be sub-linear in the domain size.

There are many scenarios in which a distribution is observed along with noise; in some cases, even the form of the noise is known \emph{a priori}.  
One such scenario is  the so-called {\em identity testing} problem in which the tester has a known (explicitly specified) distribution and its goal is to check if a given distribution, available as samples, is close to the known distribution.  For example, assume that the distribution of the top million queries to a web search engine is known in advance.  Then, identity testing would be a quick way to check how close the daily query distribution is to this known distribution.  However, in reality, there are natural minor variations to the daily query distribution, which may cause the identity tester to fail.  This is clearly undesirable.  

An option to tackle the noise would be to use testers that are {\em tolerant} to noise.  Unfortunately, even simple versions of tolerant testers are faced with near-linear lower bounds on the sample complexity, making this option uninteresting.  For example, one can distinguish if a distribution on a domain of size $n$ is uniform or far from uniform in $\ell_1$-distance using $\Theta(\sqrt{n})$ samples~\cite{Paninski:08}. However, an algorithm that distinguishes between near-uniform distributions and  distributions that are far from uniform requires $\Omega(n/\log n)$ samples~\cite{Valiant08,ValiantV14}.  Hence, to achieve sub-linear sample complexity, we need more judicious, stylized assumptions about the noise---how it is available to the tester and if it is adversarial.

A different yet natural way to model the above scenario is to view it as a mixture of distributions.
In the above example, one of the components of the mixture can be interpreted as the signal and the other component can be thought of as the noise.
More generally, the tester is given the components of a mixture of two distributions.   However, it does not know the mixing parameter, i.e., the magnitude of the contribution of each component to the mixture.  The mixture testing problem is then to test if a distribution is close to a mixture of two given distributions or is far from any manner in which the two distributions can be mixed.  As we will see, by making reasonable assumptions on the form of the noise and how it is available to the tester, the tolerant testing lower bounds can be circumvented and one can obtain testers with sub-linear sample complexity.

\paragraph{Main contributions.}
In this work, we consider distribution testing of mixtures of two distributions $q_1$ and $q_2$.  
For ease of exposition, let us call the first component $q_1$ the \emph{original distribution} and the second component $q_2$ the \emph{noise}, and let $[n] = \{1, \ldots, n\}$ be the domain of both $q_1$ and $q_2$.
The simplest version of our problem is: given sample access to distribution $p$, and for known distributions  $q_1,q_2$,  is $p = \alpha q_1 + (1- \alpha) q_2$ for some $\alpha$, or is $p$ far from $\alpha q_1 + (1 - \alpha) q_2$ for every $\alpha \in [0, 1]$?
Note that the tester is not given the mixture parameter $\alpha$.  
We further study the case when $q_1,q_2$ are not given explicitly
to the algorithm, as well as other generalizations.

We mainly focus on identity and closeness testing, which are two basic instances of hypothesis testing that have received much attention in the theory, machine learning, and statistics communities; see the works of~\cite{GGR98, BatuFFKRW, BFRSW, Batu01, BatuDKR02, BatuKR04, Paninski:08, Valiant08, GR00, ILR12, LRR13, DaskalakisDSV13, AcharyaJOS14c,   ChanDVV14, colt2, DBLP:journals/corr/DiakonikolasKN14, DiakonikolasKN15, ADK15, colt5, CDGR16, colt1, colt4, ValiantV14, nips1, nips2, colt3} and the surveys of~\cite{Rub12} and~\cite{canonne2015survey}.

The mixture testing problem has a more constrained model compared to the tolerant testing problem, so one might hope to bypass the existing lower bounds. However, the mixture testing problem can also run into near-linear sample complexity lower bounds if one does not provide the tester with sufficient access to the mixture components. 
Indeed, if the tester does not have access to the noise, we show the mixture testing problem becomes as hard as tolerant testing, necessitating $\Omega(n/\log n)$ samples (Theorem~\ref{thm:LB_mixture_no_info}).  
Hence, to show nontrivial positive results, 
the tester must have access to some kind of information 
about the noise.
We consider the following three cases for the noise, namely, (i) when the noise is given as an explicitly specified distribution, (ii) when the tester does not explicitly know the noise distribution,  but does have sample access to it, and (iii) when there is no explicit description or access to samples from the  noise distribution, but it is known that the noise distribution comes from a {\em class} of distributions, e.g., the set of  $k$-histogram distributions. 
For the first, we obtain a tester with sample complexity $\Theta(\sqrt{n}/\epsilon^2)$ and for the second, we obtain a tester with sample complexity $\Theta(\sqrt{n}/\epsilon^2 + n^{2/3}/\epsilon^{4/3})$  where $\epsilon$ is a given proximity parameter; these show that the complexity of our testers is exactly the same as for the classical non-mixture case.  For the third, when the noise is assumed to come from the set of $k$-histogram distributions, 
we obtain an identity tester that uses 
$\Tilde{O}(\sqrt{k n})$ samples.

\section{Preliminaries}

For the rest of the paper, we use the following notation.  For a distribution $p$ over $[n]$, we use $p(i)$ to denote the probability of element $i \in [n]$ and for a subset $S \subseteq [n]$, let $p(S) = \sum_{i \in S} p(i)$.   
We use $\|.\|_p$ to indicate the $\ell_p$-norm of a vector.
We typically use the $\ell_1$-distance and say $p$ and $q$ are \emph{$\epsilon$-close} if $\|p-q\|_1 < \epsilon$ and \emph{$\epsilon$-far} otherwise. Let $\UU_n$ denote the uniform distribution on $[n]$; 
we drop the subscript when the domain is clear from the context. 
Distribution $p$ is a {\em mixture} of $q_1$ and $q_2$ if there exists $\alpha \in [0,1]$ such that $p = (1-\alpha)\,q_1 + \alpha q_2$. 
We call $\alpha$ the {\em mixture parameter}. We use $q_\alpha$ to denote the mixture $(1-\alpha)q_1 + \alpha q_2$ when the components $q_1$ and $q_2$ are clear from the context.

\paragraph{Background.}

Through this paper we consider several {\em distribution testing} problems: 
For a given  {\em property} of distributions, 
we use $\PP$ to denote a set of distributions that satisfy 
the property. The distance of distribution $p$ to $\PP$  is the $\ell_1$-distance between $p$ and the closest  distribution $q$ in $\PP$. 
In a {\em distribution testing problem}, the goal is to distinguish whether $p$ is in $\PP$ or is $\epsilon$-far from $\PP$. 
We say an algorithm is a {\em tester for property $\PP$}  if the following is true with probability $2/3$. \footnote{The success probability of $2/3$ is arbitrary here. 
Given such tester, we can achieve a success probability of $1-\delta$, via standard amplification methods, at the cost of a $\log(1/\delta)$ multiplicative increase in the sample complexity.}
\begin{itemize}[nosep]
    \item Completeness: If $p$ is in $\PP$, then the algorithm outputs \accept. 
    \item Soundness: If $p$ is $\epsilon$-far from $\PP$, the algorithm outputs \reject. 
\end{itemize}
The  algorithm is an \emph{$(\epsilon', \epsilon)$-tolerant tester}, if it 
also satisfies the stronger completeness property that when $p$ 
is $\epsilon'$-close to some distribution in $\PP$, then the algorithm outputs \accept (with probability at least $2/3$).
These definitions can be extended to the case of properties of collections of more than one distribution. 
Although in the standard setting we receive samples from at least one distribution in the collection, the testing problems may be 
defined with respect to other methods of access.

We make one of the three following assumptions regarding the algorithm's view of the distributions: (i) The distribution is {\em explicitly given} or {\em known} if the algorithm knows the probability of each domain element under the distribution. 
(ii) The distribution is {\em given by samples} if the algorithm has access to an oracle that provides samples from the distribution.
(iii) The distribution is not known nor given by samples but is a member of a given class of distributions.

The term {\em  identity testing} is used to refer to the setting in which we test if a distribution, which we have sample access to, is equal to a known one.
Note that this is equivalent to testing property
$\PP = \{ q \}$.
The term {\em closeness testing } refers to the setting in which we test if two 
distributions, both available via samples, are equal or not; in this case, $\PP$ is the set of pairs of equal distributions. 

\paragraph{Mixture testing problems.}
Suppose $p$, $q_1$, and $q_2$ are distributions over $[n]$. 
Let 
$\PP_{q_1,q_2} \coloneqq \{(1-\alpha) \, q_1 + \alpha \, q_2 ~\mid~ \alpha \in [0,1] \}$ (we usually drop the subscripts $q_1,q_2$ when  
they are clear from context).
In a {\em mixture testing problem}, the goal is to distinguish
whether a distribution $p$  given via samples is in $\PP_{q_1,q_2}$
or $\epsilon$-far from any distribution in $\PP_{q_1,q_2}$ 
with probability at least 2/3.
We investigate the following problems, which differ in the way
that mixture testing algorithm can access $q_1,q_2$.  
Note however that the mixture parameter $\alpha$ is \emph{not} given
to the tester.
(i)  An algorithm is an {\em identity tester in the presence of known noise} 
if it solves the mixture testing problem when $q_1,q_2$ are
known to the tester.  
(ii) An algorithm is a 
{\em closeness tester in the presence of noise that is accessible via samples 
}
if it solves the mixture testing problem when $q_1,q_2$ are 
not explicitly given, but samples of each
are provided to the tester.
(iii) An algorithm is an
{\em identity tester in the presence of  class $\CC$-noise} if it can distinguish whether $p$ is a mixture of a known distribution $q_1$ and some $q_2 \in \CC$.   
Note that such an algorithm is a property tester for  $\PP \coloneqq \{(1-\alpha)\,  q_1 + \alpha \, q_2 ~\mid~ q_2 \in \CC \mbox{, and }\alpha \in [0,1] \}$.
%


Note that one can also define ``closeness testing in the presence
of known noise'', and ``identity testing in the presence of noise 
that is given via samples'',
but our lower bounds will show that the sample complexity of 
these tasks is the same as the sample complexity of 
closeness testing in the presence of noise that is given via samples.

\section{An overview of our results and techniques} 

\subsection{Testing identity in the presence of known noise} 

We first consider the problem of testing if distribution $p$, given
via samples, can be expressed as mixture of known distributions $q_1$ and $q_2$.  We show the following.

\begin{theorem} \label{thm:int_identity}
Given two known distributions $q_1$, $q_2$, 
and $\epsilon > 0$, 
there is an identity tester in the presence of known noise
that uses $ O(\sqrt n/\epsilon^2)$ samples.
Furthermore,  $\Omega(\sqrt n/\epsilon^2)$ samples are
required.
\end{theorem}

At a high level, we take the following steps to test if $p$ is a mixture or $\epsilon$-far from it. First, 
we develop an algorithm ({\em learner}) to learn mixture distributions. 
The learner receives samples from $p$ and outputs a mixture distribution $q_\alpha$. 
If $p$ is a mixture, then we show that the learner finds a mixture distribution $q_\alpha$ that is $\epsilon'$-close to $p$ for some proximity parameter $\epsilon' \coloneqq \Theta(\epsilon)$; 
and if $p$ is not a mixture, the learner outputs $q_\alpha$  with no specific guarantee. 
Second, we use the distance between $p$ and $q_\alpha$ as a measure to decide about $p$: if $p$ is $\epsilon'$-close to $q_\alpha$, we accept $p$; and if $p$ is $\epsilon$-far from $q_\alpha$, we reject it. This approach results in a tester for $p$. In fact,
if $p$ is a mixture, then we show that the learner 
finds a $q_\alpha$ that is $\epsilon'$-close and if $p$ is $\epsilon$-far from being a mixture, 
then we show that $p$ has to be $\epsilon$-far from any mixture distribution, including $q_\alpha$.

The challenge in this approach 
is to distinguish whether $p$ is close to $q_\alpha$ or far from it.
In general, testing whether two distributions are $\epsilon'$-close or $\epsilon$-far from each other requires ${\Omega}(n/\log n)$ samples.
However, we show that we can
exploit the  structural properties of 
mixture distributions to achieve a sample-efficient algorithm. 
Below we provide a more detailed description of the steps.

\paragraph{The learner.} 
The algorithm begins by assuming that the given distribution $p$ 
is indeed a mixture, and attempts to learn the mixture parameter:
If $p$ is a mixture, then 
we show that it can be learned to 
error $\epsilon' = \Theta(\epsilon)$ using  $O(1/\epsilon^2)$ samples given $q_1$ and $q_2$. 
The algorithm picks a subset $S$ of elements such that it  contains every element $i$ for which $q_1(i) \geq q_2(i)$ and estimates the weight of these elements according to $p$, i.e., $p(S)$.
$S$ satisfies that $q_1(S) - q_2(S)$ is exactly the total variation distance between $q_1$ and $q_2$. 
Comparing $p(S)$ with the weight of these elements according to $q_1$ and $q_2$ guides us to choose a mixture parameter $\alpha$, 
and allows us
to bound the distance between $p$ and $q_\alpha \coloneqq (1-\alpha)q_1 + \alpha q_2 $. 
(Instead of learning $\alpha$, one might do a grid search on $\alpha \in [0, 1]$; however the granularity required could make the resulting algorithm sub-optimal.)

\paragraph{Assessing the distance between $p$ and $q_\alpha$.} 
After obtaining $q_\alpha$, the task of 
distinguishing whether distribution $p$ is a 
mixture or $\epsilon$-far 
from a mixture boils down to testing if $p$ is $\epsilon'$-close to $q_\alpha$ or is $\epsilon$-far from it. We propose a scheme to \emph{reshape} the distributions $p$ and $q_\alpha$ and get two new distributions $p'$ and $q'_\alpha$ such that for $p$ that is  a mixture, the  $\ell_2$-distance between $p'$ and $q'_\alpha$ is at most $O(\epsilon/\sqrt n)$. Furthermore, in the case where $p$ is $\epsilon$-far from being a mixture, $p'$ is $\epsilon$-far from $q'_\alpha$. It is known that one can efficiently distinguish the case that $\|p' -  q'_\alpha\|_2 \leq O(\epsilon/\sqrt n)$ versus $\|p' - q'_\alpha\|_1 \geq \epsilon$ using $O(\sqrt{n}/\epsilon^2)$ samples \cite{
DiakonikolasK:2016, ChanDVV14}.

Here, we elaborate further on how we reshape the distributions.
Similar techniques have been used previously
to reduce the $\ell_2$-norm,
e.g., in~\cite{DiakonikolasK:2016}. 
Here, we use it to bound the $\ell_2$-distance between $p'$ and $q'_\alpha$. The reshaping process is as follows.  Define  $p'$, the reshaped distribution of $p$ with a new domain which is larger than the domain of $p$. For each element $i$,  we determine an integer $a_i$ solely based on $q_1$, $q_2$, and $q_\alpha$. Then we add $(i,j)$ for all $j$ in $[a_i]$ to the domain of $p'$. We set the probability of element $(i,j)$ to be $p(i)/a_i$. Also, we reshape $q_\alpha$ according to the same process and get $q'_\alpha$. 

But how can reshaping reduce the $\ell_2$-distance? 
Given that $p$ is a mixture, for each element $i$ in the domain, 
the discrepancy between the probability of $i$ according to $p'$ and $q'_\alpha$, $|p(i) - q_\alpha(i)|/a_i$, is proportional to $|q_1(i) - q_2(i)|/a_i$. 
With this observation, we set the $a_i$'s such that 
they make the discrepancy $O(\epsilon/n)$ for each element.
This ensures the $\ell_2$-distance  between $p'$ and $q'_\alpha$ is  $O(\epsilon/\sqrt{n})$.


The arguments described above are formalized in Theorem~\ref{thm:UB_identity}. In addition, in the case where $q_1$ and $q_2$ are uniform, this problem is as hard as testing if a distribution is uniform, 
which needs at least $\Omega(\sqrt{n}/\epsilon^2)$ samples
(\cite{Paninski:08}),
showing that the sample complexity of our algorithm is tight.
Furthermore, we match the sample complexity of the 
standard identity tester where there is no noise involved.

\subsection{Testing closeness in the presence of noise 
that is accessible via samples}
We next investigate the problem of testing closeness of distributions in the presence of noise 
that is accessible via samples. Suppose we have sample access to three distributions, $p$, $q_1$, and $q_2$, over $[n]$. The goal is to test if there is a mixture parameter $\alpha^*$ such that $p = (1-\alpha^*)\,q_1  + \alpha^* q_2$, or $p$ is $\epsilon$-far from any distribution in this form. 

Similarly to the identity testing algorithm explained earlier,  our approach is first attempt to learn $p$. 
That is, we design an algorithm that finds a candidate mixture distribution, $q_\alpha \coloneqq (1-\alpha)q_1 + \alpha\,q_2$, such that if $p$ is a  mixture of $q_1$ and $q_2$, then $p$ and $q_\alpha$ will be $(\epsilon/\sqrt n)$-close to $p$ in $\ell_2$-distance; and if $p$ is not a mixture, the algorithm finds a distribution $q_\alpha$ with no specific guarantees. 
Then, we test to see if $p$ is $(\epsilon/(2\sqrt{n}))$-close to $q_\alpha$ in $\ell_2$-distance, or $(\epsilon/\sqrt{n})$-far from it.
The answer of the test dictates if we should accept or reject $p$.  Indeed, if $p$ is a mixture distribution, 
$q_\alpha$ is very close to $p$, and the test will accept $p$.
If $p$ is $\epsilon$-far from being a mixture, 
then $p$ is $\epsilon$-far from 
$q_\alpha$, and furthermore $p$ and 
$q_\alpha$ are $(\epsilon/\sqrt{n})$-far from each 
other in $\ell_2$-distance, so that the test will reject $p$. 


But how do we learn $p$? Since we are looking for $q_\alpha$, which is close to $p$ in $\ell_2$-distance, 
we study the problem of estimating the $\ell_2$-distance between $p$ and a mixture distribution of $q_1$ and $q_2$. Inspired by the $\ell_2$-distance estimator proposed by \cite{ChanDVV14}, we propose a statistic such that given $\alpha$ it estimates the $\ell_2$-distance between $p$ and $q_\alpha$:
$f(\alpha) := \sum\limits_{i=1}^n (X_i - (1-\alpha)Y_i - \alpha Z_i)^2 - X_i - (1-\alpha)^2 Y_i - {\alpha}^2 Z_i,$
where $X_i$, $Y_i$, and $Z_i$ are the number of instances 
of element $i$ among samples from $p$, $q_1$, and $q_2$ respectively.
The statistic is designed such that it is equal to  $s^2 \|p - q(\alpha)\|_2^2$ in expectation where $s$ is the number of samples from each distribution $p$, $q_1$, and $q_2$.

Given the sample sets, the goal is to use the quadratic function $f$ to find a candidate $\alpha$. For now, assume $p$ is a mixture of $q_1$ and $q_2$ with parameter $\alpha^*$. We make two observations about $f(\alpha^*)$: 
(i) the expectation of $f(\alpha)$ is minimum, in fact zero, when $\alpha = \alpha^*$, and
(ii) we provide a threshold $T$ for which $|f(\alpha^*)|$ is at most $T$ with high probability. 
Although $\alpha^*$ is not given to the algorithm, 
we wish to pick a  candidate $\alpha$ that is very 
close to $\alpha^*$. 
We use the above two observations as a guide to take the following strategy: find  $\alpha$ that minimizes $f(\alpha)$ while $|f(\alpha)|$ is at most $T$.  This method apparently finds several candidate $\alpha$'s.
We establish that if $p$ is a mixture, 
then one of the candidate $\alpha$'s  will result in a mixture distribution $q_\alpha$ that is $(\epsilon/2\sqrt{n})$-close to $p$ in $\ell_2$ distance.  (Once again, a grid search on $\alpha \in [0, 1]$ will not yield an optimal sample complexity.)

From there on, we only need to test if any of 
the candidates we found are $(\epsilon/2\sqrt{n})$-close to $p$ or not. If $p$ is a mixture we are promised that one of the candidates will pass the test. 
Otherwise we show that all candidates have to give distributions that are $\epsilon$-far (implying $(\epsilon/\sqrt {n})$-far in $\ell_2$-distance) from $p$ by definition, so all of them will fail. Our approach yields the following result: 

\begin{theorem} Assume we have sample access to three distributions $p$, $q_1$, and $q_2$ over $[n]$. 
There exists a closeness tester in the presence of noise
that uses $O(\sqrt{n}/\epsilon^2 + n^{2/3}/\epsilon^{4/3})$ samples. 
Furthermore
$\Omega(\sqrt{n}/\epsilon^2 + n^{2/3}/\epsilon^{4/3})$ samples
are required. 
\end{theorem}
See Theorem~\ref{thm:UB_mixture_closeness} for the formal statement of the result. 
For the lower bound of sample complexity, we establish that the lower bound for standard closeness testers holds in the mixture setting as well, even in the case where $q_1$ or $q_2$ is known. In particular, we show given sample access to $p$ and $q_1$, testing whether $p$ is a mixture of $q_1$ and the uniform distribution requires $\Omega(\sqrt{n}/\epsilon^2 + n^{2/3}/\epsilon^{4/3})$ samples 
(Proposition~\ref{prop:lb_unknown_q_U}).
Hence, one cannot hope to achieve a better sample complexity.

\subsection{Testing identity in the presence of \texorpdfstring{$k$}{k}-flat noise} 


On the one hand, the sample complexity
of distribution testing 
under arbitrary noise is significantly 
worse than that of noise-free distribution testing. On the other hand, we have seen that
the sample complexity of distribution testing with noise (either known or given via sample access) is very similar to the
sample complexity of noise-free distribution testing.
This raises the question of whether one can relax the requirement of
the access to the noise  by the tester and still achieve better sample complexity.  
The next problem we consider is the identity testing problem
when there is no direct access to the noise (either via samples,
or an explicit description) except for the promise that
the noise comes from a class, in particular, the class of $k$-flat distributions.

We say a distribution is \emph{$k$-flat} if the probability mass function of the distribution is a piece-wise constant function with $k$ pieces.  We investigate the following problem: given a known distribution $q$ and having sample access to $p$, 
can we distinguish if $p$ is a mixture of $q$ and some $k$-flat distribution, 
or $p$ is $\epsilon$-far from any such distribution? 
We provide an algorithm that uses $\Tilde{O}( \sqrt {kn})$ samples.

Inspired by the identity tester proposed in \cite{BatuFFKRW}, we propose the following approach.
First, we guess the $k$ intervals on which the noise is constant.
Then, we take the elements of each interval 
and further partition them into subsets (not necessarily contiguous)
such that in each 
subset the probability of the elements according to 
$q$ are very similar to each other (similar enough so that 
we can show that
$q$ is nearly-uniform on each subset).
For a mixture distribution $p$, if we have guessed the intervals correctly, 
$p$ is almost uniform within each subset since it is a mixture of an almost uniform $q$ and a constant function (noise). 
Hence to see if $p$ is a mixture, we first test each of these subsets and see if $p$ is close to uniform on them.   
We then estimate the total weights that $p$ assigns to each
of these subsets and determine if the weights are
consistent with a mixture of $q$ and some $k$-flat distribution.
One challenge is to find a sampling method that guarantees
good results for all initial guesses of the $k$ intervals describing
the noise. See Section~\ref{sec:kFlat} for details. 

It is not hard to see that in the case where $q$ is uniform, identity testing in the presence of $k$-flat noise is as hard as testing whether a distribution is $k$-flat or $\epsilon$-far from any such distribution. Therefore, there exists a lower bound of $\Omega (\sqrt{n}/\epsilon^2 + k/(\epsilon \, \log k) )$ for our problem derived from the lower bound for testing $k$-flat distributions in \cite{Canonne16}.

\subsection{Lower bounds} 
We show that testing identity with respect to the uniform distribution when the noise component can be an arbitrary distribution, requires near-linear samples, i.e., $\Omega(n/\log n)$. More specifically, 

\begin{restatable*}{theorem}{LBBigness} \label{thm:LB_mixture_no_info}
Assume $p$ is a distribution on $[n]$. There exists a constant  parameter $\epsilon$ such that distinguishing the following cases with probability at least $2/3$ requires $\Omega(n/\log n)$ samples. 
\begin{itemize}[nosep]
    \item There exists a noise distribution on $[n]$, 
    namely $\eta$, and an $\alpha \leq \epsilon_1$ 
    such that $p$ is a mixture of uniform and $\eta$ with parameter $\alpha$, 
    i.e., $p = (1-\alpha) \, \UU +  \alpha \, \eta$.
    \item There is no noise distribution $\eta$  such that 
    $p = (1- \alpha) \, \UU + \alpha \, \eta$ unless $\alpha = 1$.
\end{itemize}
\end{restatable*}
The main idea is to reduce this problem to 
the that of testing the {\em $T$-bigness property}~\cite{aliakbarpourGPRY18},
which holds if 
all probabilities are above a given threshold $T$. 

\section{Identity testing of mixtures in the presence of  known noise}
\label{sec:identity}

In this section, we give an algorithm which tests if $p$ is close to a mixture where both the components are explicitly known.  As before, we assume the mixture parameter $\alpha$ is unknown. 

The main idea is  attempt to learn a mixture distribution $q_\alpha$ that is close to $p$.   Using $q_1$, $q_2$, and $q_\alpha$, we then reshape the distribution $p$ to another distribution $p'$
and use the same reshaping to transform $q_\alpha$ to $q'_\alpha$.  The reshaping has the property that in the case that $p$ is indeed a mixture, then $p'$ and $q'_\alpha$ will be extremely close to each other in $\ell_2$-distance and if $p$ is not a mixture, then $p'$ and $q'_\alpha$ will be quite far from each other.  Thus we can use a (non-tolerant) identity tester on $p'$ and $q'_\alpha$.
 
In the rest of this section, we present the three main steps of the testing algorithm.  The first step is a learner algorithm that finds $q_\alpha$ (Section~\ref{sec:learner}).  The second step is a reshaping process that transforms the distributions $p$ and $q_\alpha$ into $p'$ and $q'_\alpha$ respectively (Section~\ref{sec:reshape}).  The third step is to put these pieces together to get the identity tester (Section~\ref{sec:unif3}).

\subsection{The learner}
\label{sec:learner}

At a high level, the leaner proceeds as follows.  Observe that if $p$ is a mixture of $q_1$ and $q_2$, then there is a parameter $\alpha^* \in [0, 1]$ such that $p = (1 - \alpha^*) q_1 + \alpha^* q_2$, so to learn $p$ it is sufficient to learn $\alpha^*$. Let $S$ be the set of all domain elements, $x$,  where $q_1(x)$ is at least $q_2(x)$. 
By definition, for $S \subseteq [n]$, we have $p(S) = (1 - \alpha^*)q_1(S) + \alpha^* q_2(S)$, which leads to $\alpha^* = \left( q_1(S) - p(S) \right)/\left(q_1(S) - q_2(S)\right)$.  
The idea then is to replace $p(S)$ with its estimate, say, $w_S$ to get an estimate $\alpha$ of $\alpha^*$.  We formally describe the procedure in Algorithm~\ref{alg:learner} and prove its correctness in Lemma~\ref{lem:learner}.


\inCOLT{
\begin{wrapfigure}{L}{0.45\textwidth}
\fbox{
\begin{minipage}{0.42\textwidth}
\begin{algorithm}[H]
\DontPrintSemicolon
\SetKwInOut{Input}{Procedure}
\floatconts{alg:learner}
{\caption{Learning mixture of two \,\, \, \, known distributions.}}
{
{\vspace{2mm}}
{{\bf Procedure:} \textsc{Mixture-Learner}($q_1$,} \\
{\hspace*{1cm} $q_2$, $n$, $\epsilon$, sample access to $p$)} \\
 	\If{$\|q_1- q_2\|_1 \leq \epsilon$}{\KwRet{$0$}}
 	{$S \gets \{i \in [n] ~\mid~ q_1(i) > q_2(i)\}$}\\
 	{$m \gets O(1/\epsilon^2)$}\\
  	{$w_S \gets ({\mbox{\# samples in $S$}})/{m}$}\\
    {$\alpha \gets \dfrac{q_1(S) - w_S - \epsilon/4}{q_1(S) - q_2(S)} $}\\
    {\KwRet{$\alpha$}}
    {\vspace{3mm}}
}
\end{algorithm}
\end{minipage}
}
\end{wrapfigure}
}{
\begin{algorithm}[ht]
\caption{Learning mixture of two known distributions.}
\label{alg:learner}
\begin{algorithmic}[1]
\Procedure{Mixture-Learner}{$q_1$, $q_2$, $n$, $\epsilon$, sample access to $p$}
    \If{$\|q_1- q_2\|_1 \leq \epsilon$}
        \State{\Return {$0$}}
    \EndIf
 	\State{$S \gets \{i \in [n] ~\mid~ q_1(i) > q_2(i)\}$}
 	\State{$m \gets O(1/\epsilon^2)$}
  	\State{$w_S \gets \dfrac{\mbox{\# samples in $S$}}{m}$}
    \State{$\alpha \gets \dfrac{q_1(S) - w_S - \epsilon/4}{q_1(S) - q_2(S)} $}
    \State{\Return {$\alpha$}}
\EndProcedure
\end{algorithmic}
\end{algorithm}
}

\begin{lemma} \label{lem:learner}  Suppose $p = (1-\alpha^*)q_1 + \alpha^* q_2$. Using $O(1/\epsilon^2)$ samples, Algorithm~\ref{alg:learner} outputs a mixture parameter $\alpha \leq \alpha^*$ such that
$\Pr{ \| q_{\alpha} - p \|_1 < \epsilon} \geq 5/6$.  
\end{lemma}

\begin{proof}
First, observe that if $q_1$ and $q_2$ are $\epsilon$-close, then $p$ is $\epsilon$-close to $q_1$ as well, so distribution $q_1$ which is a mixture with parameter $\alpha = 0$ is a valid output.  For the remainder of the proof, we assume $q_1$ and $q_2$ are $\epsilon$-far from each other. 

Let $S = \{ i \in [n] ~\mid~ q_1(i) > q_2(i) \}$.  We have
$
\alpha^* = \frac{q_1(S) - p(S)}{ q_1(S) - q_2(S)}.
$
The only unknown value in the above expression is $p(S)$, which we estimate using $O(1/\epsilon^2)$ samples from $p$.
We show by replacing $p(S)$, we get a viable estimate for $\alpha^*$. 

Let $w_S$ be the estimate that is the ratio of the samples that are in $S$.  By the Hoeffding bound, we have 
$\Pr{|p(S) - w_s| \leq \epsilon/4} \geq 5/6.$ 
We define our estimate of $\alpha^*$ as:
$
\alpha \coloneqq \frac{q_1(S) - (w_S + \epsilon/4)}{ q_1(S) - q_2(S) } \,.
$
The reason that we add $\epsilon/4$ to $w_S$ is to assure an overestimation of  $p(S)$, so $\alpha$ becomes smaller than $\alpha^*$ with high probability. I.e., since with high probability $p(S) \leq w_S + \epsilon/4$, we get:
$$
    \alpha^* = \frac{q_1(S) - p(S)}{ q_1(S) - q_2(S) } 
    \geq \frac{q_1(S) - (w_S + \epsilon/4)}{ q_1(S) - q_2(S) } = \alpha.
$$
Below, we show $q_\alpha$ is close to $p$ in $\ell_1$-distance. Based on the definition of $S$, $q_1(S) - q_2(S)$ is equal to the total variation distance (i.e., half of the $\ell_1$-distance) between $q_1$ and $q_2$. With probability $5/6$, 
$$
\begin{array}{ll}
\|p - q_\alpha\|_1 & = \sum\limits_i \left|p(i) - q_\alpha(i)\right|
\enspace = \enspace \sum\limits_i \left|(1-\alpha^*) q_1(i) + \alpha^* \, q_2(i) - (1-\alpha) q_1(i) - \alpha \, q_2(i)\right|
\\ & = \sum\limits_i \left|(\alpha-\alpha^*) (q_1(i) - q_2(i))\right|
\enspace = \enspace \left|(\alpha-\alpha^*)\right| \cdot \|q_1 - q_2\|_1 \vspace{2mm} \\ 
& = 2 \left|(\alpha-\alpha^*)\right| \cdot \left(q_1(S) - q_2(S)\right) \vspace{1mm} \enspace = \enspace 2 \left|p(S) - w_S -\frac{\epsilon}{4}\right|
\enspace \leq \enspace \epsilon.
\vspace*{-0.9cm}
\end{array}
$$
\end{proof}


\vspace{-3mm}
\subsection {Reshaping the distributions} 
\label{sec:reshape}

Using Algorithm~\ref{alg:learner}, given $p$, we can obtain a mixture parameter $\alpha$ and a mixture distribution $q_\alpha$ for which (i) if $p$ is the mixture of $q_1$ and $q_2$ with parameter $\alpha^*$, then $q_\alpha$ is $\epsilon'$-close to $p$ for a proximity parameter $\epsilon' = \epsilon/6$, and $\alpha\leq \alpha^*$ and (ii) if $p$ is $\epsilon$-far from being a mixture, then $p$ is $\epsilon$-far from $q_\alpha$.  
Ideally, we wish to use an identity tester to see if  $q_\alpha$ and  $p$ are roughly the same or far from each other. 
Unfortunately, this is not possible in general, unless $p$ and $q_\alpha$ are very close on \emph{every} domain element.  
To resolve this issue, the goal in this section is to introduce two distributions $p'$ and $q'_\alpha$ such that (i) when $p$ and $q_\alpha$ are close, $p'$ and $q'_\alpha$ are very close to each other on every domain element and (ii) when $p$ and $q_\alpha$ are far, $p'$ and $q'_\alpha$ are far.  
Our reshaping process is inspired by the method of
\cite{DiakonikolasK:2016}: 
For each element $i \in [n]$, using $q_\alpha$, we define:
$$a_i \coloneqq \floor{n q_\alpha(i)} + \floor {\frac{n|q_\alpha(i) - q_2(i)|}{ \| q_\alpha - q_2 \|_1 }}+ 1.
$$
Note that the process in  \cite{DiakonikolasK:2016} 
uses only the first and third
terms of the above sum in defining $a_i$.
We start the reshaping process by associating $a_i \geq 1$ buckets to each domain element $i \in [n]$ to form a new domain $D = \{ (i, j) ~\mid~ i \in [n] $ and $j \in [a_i] \}$.  To draw a sample from $p'$, we first draw a sample $i$ from $p$, then we sample $j$ from $[a_i]$, and return the pair $(i,j)$ as the sample from $p'$.  We say $p'$ is a \emph{reshaping  of $p$} with respect to $q_{\alpha}$.  
Clearly $p'(i,j) = p(i)/a_i$. In a similar manner, we define the reshaping $q'_\alpha$ of $q_\alpha$ and once again, we have $q'_\alpha(i,j) = q_\alpha(i)/a_i$. 

We next prove several crucial properties of the reshaped distributions.
\begin{lemma} \label{lem:reshape}
Let $p'$ and $q_\alpha '$ be the result of the reshaping of $p$ and $q_\alpha$ with respect to $q_{\alpha}$ as described above. Then, the following hold:
\begin{enumerate}
\item[(i)] The $\ell_1$-distance after reshaping does not change: $\|p - q_\alpha\|_1 = \|p' - q'_\alpha\|_1$. 

\item[(ii)] The domain size of $p'$ and $q'_\alpha$, $|D| \leq 3n$. 

\item[(iii)] The $\ell_2$-norm of $q'_\alpha$,  $\| q'_{\alpha} \|_2 \leq \sqrt{3/n}$. 

\item[(iv)] If $p$ is a mixture distribution, $q_\alpha$ is $\epsilon'$-close to $p$, and $\alpha$ is at most $\alpha^*$, 
then $|p'(i,j)~-~q'_\alpha(i,j)| \leq \epsilon'/n$ for all $(i,j) \in D$.
\label{item:tiny_differences}
\end{enumerate}
\end{lemma}

\begin{proof}
To prove (i), note that
\mymath{
\|p' - q'_\alpha\|_1 
= \sum\limits_{i=1}^n \sum\limits_{j=1}^{a_i} |p'(i,j) - q'_\alpha(i,j)|
= \sum\limits_{i=1}^n \sum\limits_{j=1}^{a_i} \frac{|p(i) - q_\alpha(i)|}{a_i}
= \sum\limits_{i=1}^n |p(i) - q_\alpha(i)|
= \|p- q_\alpha\|_1.
}
For (ii), we have:
$$
\begin{array}{ll}
|D|  & = \sum\limits_{i=1}^n a_i \leq \sum\limits_{i=1}^n \left(n q_\alpha(i) + \frac{n|q_\alpha(i) - q_2(i)|}{\|q_\alpha - q_2\|_1} + 1 \right)
\\ & = n \left (\sum\limits_{i=1}^n q_\alpha(i) \right) + \frac{n}{\|q_\alpha - q_2\|_1} \left (\sum\limits_{i=1}^n |q_\alpha(i) - q_2(i)| \right) + n = 3n\,.
\end{array}
$$
We now prove (iii).  If $q_\alpha(i) < 1/n$, then $q'_\alpha(i, j) < 1/n$.  If $q_\alpha(i) \geq 1/n$, then $a_i \geq n q_\alpha (i)$ and hence $q'_\alpha(i, j) \leq 1/n$.  Therefore, $q'_\alpha(i,j) \leq 1/n$ for all $i, j$. Since the domain size of $q'_\alpha$ is at most $3n$,  
$\|q'_\alpha\|_2$ is at most $\sqrt{3/n}$. 

Finally, we show (iv).  Since $p$ is a mixture distribution, there is an $\alpha^* \in [0, 1]$ such that $p = (1-\alpha^*)q_1 + \alpha^*\, q_2$.  Also, we have that $q_\alpha$ has a mixture parameter $\alpha \leq \alpha^*$.  Furthermore, we also have $\|p - q_\alpha\|_1 \leq \epsilon'$. Let $\beta = (\alpha^* - \alpha)/(1-\alpha)$ which is in $[0,1]$. Observe that
\mymath{
(1-\beta) q_\alpha + \beta \, q_2 = (1-\beta)(1-\alpha) q_1 +\left(\alpha (1-\beta) + \beta\right)q_2 = (1 - \alpha^*) q_1 + \alpha^* \, q_2 = p \,.
}
Thus, $p$ is a mixture of $q_\alpha$ and $q_2$. For an element $(i,j) \in D$, we can bound the difference of $p'(i,j)$ and $q'_\alpha(i,j)$ as follows, which finishes the proof.
$$
\begin{array}{ll}
\left| p'(i,j) - q'_\alpha(i,j)\right| & = \frac{\left| p(i) - q_\alpha(i)\right|}{a_i}  
= \frac{\beta \cdot |q_\alpha(i) - q_2(i)|}{a_i} 
\leq \frac{\beta\cdot d\cdot  |q_\alpha(i) - q_2(i)| } {n \cdot |q_\alpha(i) - q_2(i)|} = \frac{\beta d}{n}
\\ & = \frac{1}{n} \sum\limits_{i=1}^n \beta |q_\alpha(i) - q_2(i)|
= \frac{1}{n} \sum\limits_{i=1}^n |p(i) - q_\alpha(i)| 
= \frac{\|p- q_\alpha\|_1}{n} \leq \frac {\epsilon'}{n}.
\end{array}
\vspace*{-.9cm}
$$
\end{proof}


\subsection{The mixture testing algorithm}
\label{sec:unif3}

In this section, we use the learner and the reshaped distributions to obtain an identity tester for mixtures of two known distributions. 
\inCOLT{
\begin{wrapfigure}{L}{0.45\textwidth}
\fbox{
\begin{minipage}{0.42\textwidth}
\begin{algorithm}[H]
\DontPrintSemicolon
\SetKwInOut{Input}{Input}
\floatconts{alg:mixture_tester}
{\caption{Identity tester in the \, \, \, \, \, \, presence of known noise.}}
{
    {\vspace{2mm} \textbf{Procedure: } \textsc{Mixture-Idenitty-} \, \, \,}
    {\hspace{1cm}  \textsc{Tester}($q_1, q_2, n, \epsilon$, sample access to $p$)}
    {$\epsilon' \gets \epsilon/6$}\\
    {$\alpha \gets ${\textnormal{\textsc{Mixture-Learner}}}$(q_1, q_2, n, \epsilon')$}\\
    \For {$i = 1$ \KwTo $n$}{
        {$a_i \gets \floor{n q_\alpha(i) } + \floor{\frac{n\vert q_\alpha(i) - q_2(i)\vert }{\|q_\alpha - q_2\|_1}} + 1$ }\\
        {$q'_\alpha (i) = \frac{p(i)}{a_i}$}
    }
    {$x_1, \ldots, x_s \gets \Theta\left(\frac{\sqrt{n}}{\epsilon^2}\right)$ samples from $p$}\\
	\For{$i = 1$ \KwTo $s$} {
	    {$r \gets $ uniform at random in $[a_i]$}\\
	    {$y_i \gets (x_i, r)$ }
	}
	{ \KwRet{\textsc{Identity-Tester}($q'_\alpha$, $\epsilon$, $\{ y_1, \ldots, y_s \}$ as samples from $p'$)}}
	{\vspace{2mm}}
}
\end{algorithm}
\end{minipage}
}
\end{wrapfigure}
}{
\begin{algorithm}[ht]
\caption{Identity tester in the presence of known noise.}
\label{alg:mixture_tester}
\begin{algorithmic}[1]
\Procedure{Noise-Tolerant-Identity-Tester}{$q_1, q_2, n, \epsilon$, sample access to $p$}
    \State{$\epsilon' \gets \epsilon/6$}
    \State{$\alpha \gets ${\textnormal{\textsc{Mixture-Learner}}}$(q_1, q_2, n, \epsilon')$}
    \For {$i = 1, \ldots, n$}
        \State{$a_i \gets \floor{n q_\alpha(i) } + \floor{\frac{n\vert q_\alpha(i) - q_2(i)\vert }{\|q_\alpha - q_2\|_1}} + 1$ }
        \State{$q'_\alpha (i) = \frac{p(i)}{a_i}$}
    \EndFor
    \State{$s \gets \Theta\left(\frac{\sqrt{n}}{\epsilon^2}\right)$}
    \State{$x_1, \ldots, x_s \gets s$ samples from $p$}
	\For{$i = 1, \ldots, s$} {
	    \State {$r \gets $ uniform at random in $[a_i]$}
	    \State{$y_i \gets (x_i, r)$ }
	}
	\EndFor
	\State{\Return{ \textsc{Identity-Tester}($q'_\alpha$, $\epsilon$, $\{ y_1, \ldots, y_s \}$ as samples from $p'$)}}
\EndProcedure
\end{algorithmic}
\end{algorithm}
}
\begin{restatable}{theorem}{thmIdentity} \label{thm:UB_identity} Given a proximity parameter $\epsilon$, Algorithm~\ref{alg:mixture_tester} is identity tester in the presence of known noise that uses $O(\sqrt n/\epsilon^2)$ samples. 
\end{restatable}
\begin{proof}\;  Let $\epsilon' = \epsilon/6$. In the completeness case, $p$ is a mixture distribution with parameter $\alpha^*$. Therefore, with probability at least 5/6, $q_\alpha$ is a mixture distribution with parameter $\alpha$ for which $q_\alpha$ is $\epsilon'$-close to $p$. Let $p'$ and $q'_\alpha$ be the reshaped distributions described in Section \ref{sec:reshape}.  By Lemma \ref{lem:reshape}(ii), $|D| \leq 3n$. Moreover, for any $(i,j) \in D$, $|p'(i,j)-q'_\alpha(i,j)| \leq \epsilon'/n$, which implies that 
$$
\|p' - q'_\alpha\|_2 \leq \sqrt {\frac{|D| \cdot \epsilon'^2}{n^2}} \leq \sqrt{\frac{\epsilon^2}{12\,n}} \leq \frac{\epsilon}{2\sqrt{|D|}}\,.
$$
Conversely, if $p$ is $\epsilon$-far from being a mixture distribution, then it has to be $\epsilon$-far from $q_\alpha$. By Lemma \ref{lem:reshape}, 
$p'$ and $q'_\alpha$ are $\epsilon$-far from each other. Therefore, $\|p'- q'_\alpha\|_2^2 \geq \epsilon/\sqrt {|D|}$.
Using the 
identity tester (\textsc{Identity-Tester}) provided in \cite{DiakonikolasK:2016} (see Remark 2.7 and Remark 2.8), there exists an algorithm that can distinguish the above cases with probability $5/6$ using $O(\sqrt{|D|}/\epsilon^2)$ samples. 
Thus, with probability $2/3$ both the invoked learner and the tester returns the right answer. Also, the sample complexity is $O(\sqrt{n}/\epsilon^2 + 1/\epsilon^2) = O(\sqrt{n}/\epsilon^2)$.
Hence the proof is complete.
\end{proof}
\vspace{-3mm}

\section{Testing mixtures in the presence of noise that is accessible via samples} \label{sec:closeness}

In this section, we provide an algorithm for the testing closeness of distributions in the presence of noise that is accessible via samples.  We assume we have sample access to three distributions $p$, $q_1$, and $q_2$, over $[n]$ and the goal is to test if $p$ is a mixture of $q_1$ and $q_2$.  Our approach is first to learn $p$ in an indirect manner.  Specifically, we design an algorithm that finds a candidate mixture distribution $q_\alpha \coloneqq (1-\alpha)q_1 + \alpha\,q_2$ such that with high probability if $p$ is a  mixture of $q_1$ and $q_2$, then $q_\alpha$ will be close to $p$.  We claim that the answer to the test ``is $p$ close to $q_{\alpha}$'' can be used to test if $p$ is close to a mixture of $q_1$ and $q_2$.  Indeed, if $p$ is a mixture distribution, by the property of the learning algorithm, $q_{\alpha}$ is close to $p$ and hence the test will accept.  Conversely, if $p$ is far from being a mixture, then $p$ is far from any mixture distribution including $q_\alpha$, and hence the test will reject.

In particular, the candidate $q_{\alpha}$ will be such that (i) if $p$ is a mixture, then $\| p - q_\alpha \|_2 \leq c \epsilon/\sqrt{n}$ for a sufficiently small constant $c$ and (ii) if $p$ is $\epsilon$-far from being a mixture, then $\|p - q_\alpha\|_2 \geq \epsilon/\sqrt{n}$.  
As we will see, the robust $\ell_2$-distance tester of \cite{ChanDVV14} can efficiently distinguish these two cases.  
Since we are looking for $q_\alpha$ that is close to $p$ in $\ell_2$-distance, we study how we can estimate the $\ell_2$-distance between $p$ and a mixture distribution $q_1$ and $q_2$.  
Let $s$ be the expected number of samples we draw; $s$ will be specified later.  
Assume we draw $\poi(s)$ samples%
\footnote{$\poi(s)$ is a Poisson random variable with parameter $s$.}
from $p$, $q_1$, and $q_2$. Let $X$, $Y$, and $Z$ denote the (multi)set of samples from $p$, $q_1$, and $q_2$ respectively.  Let $X_i$, $Y_i$, and $Z_i$ be the numbers of instances of element $i \in [n]$ in each sample set.  Consider the following statistic:%
\footnote{
This is motivated by the $\ell_2$-distance estimator proposed in \cite{ChanDVV14} in which they draw a set of samples from $p$ and $q$ and use the statistic $\sum_{i=1}^n (X_i - Y_i)^2 - X_i - Y_i$, where $X_i$ (resp., $Y_i$) is the number of times $i \in [n]$ occurs in the samples from $p$ (resp., $q$). 
}

\vspace{-3mm}
\begin{equation}\label{eq:cls_stat}
f(\alpha) := \sum_{i=1}^n (X_i - (1-\alpha)Y_i - \alpha Z_i)^2 - X_i - (1-\alpha)^2 Y_i - {\alpha}^2 Z_i.
\end{equation}
\vspace{-3mm}

\noindent
Note that if we fix the sample sets $X$, $Y$, and $Z$, $f$ is a quadratic function of $\alpha$.  We show that the above statistic has the expected value of $s^2 \|p - q_{\alpha}\|_2^2$. 

If $p$ is a mixture distribution with parameter $\alpha^*$, then $\E{f(\alpha^*)} = 0$, where the expectation is taken over the randomness of the samples.  Hence, a natural candidate to approximate $\alpha^*$ is some  $\alpha$ where $f$ achieves its (near-)minimum.  To do this, we first show that if $p$ is a mixture, then we can choose a threshold parameter $T$ such that $|f(\alpha^*)| \leq T$ with high probability. Then we pick $\alpha \in [0,1]$ that minimizes $f(\alpha)$ with the constraint that $|f(\alpha)| \leq T$.  Since $f$ is a quadratic, let $\hat \alpha_1$ and $\hat \alpha_2$ be the solutions.  We then show that if $p$ is a mixture of $q_1$ and $q_2$, with high probability, at least one of $\|p - q_{\hat \alpha_1}\|_2$ or $\|p - q_{\hat \alpha_2}\|_2$ is small (Section~\ref{sec:prp_candidates}).  

For the rest of the section, let $b$ be a parameter specified later for which $\|p\|_2^2$, $\|q_1\|_2^2/2$, and $\|q_2\|_2^2/2$ are bounded by $b$.  Thus for any mixture distribution $q_\alpha$, we have $\| q_{\alpha} \|^2_2 \leq b$.  Also, let $\gamma = \epsilon^2/(10 n)$,
$s = c_s \cdot (\sqrt b/(\gamma/2))$ for a sufficiently large constant $c_s$, and let $T = s^2 \gamma$.  (All missing proofs are in Section~\ref{sec:closeness_lemmas}.)


\subsection{Finding candidates} \label{sec:prp_candidates}

In this section, we aim to learn a mixture distribution. More precisely, we are looking for $\alpha$'s such that if $p$ is a mixture distribution, then with high probability, $\| p - q_\alpha \|_2 \leq \epsilon/(2\sqrt{n})$. %
\begin{theorem}\label {thm:proper_candidates} Suppose $p$ is a mixture distribution. Given $X$, $Y$, and $Z$, with probability $0.94$, one can compute a candidate set $\MM, |\MM| \leq 5$, for which there exists $\alpha \in \MM$ such that
$\|p - q_\alpha \|_2^2 \leq \frac{\epsilon^2}{4n}\,.$
\end{theorem}

\begin{proof}\; We consider two cases based on the $\ell_2$-distance of $q_1$ and $q_2$.  Suppose $\|q_1 - q_2\|_2^2 \leq \epsilon^2/(4\,n)$. If $p$ is a mixture distribution with parameter $\alpha^*$, then we have:
\vspace{-4mm}
\begin{align*}
\|q_1 - p\|_2^2 & = \sum\limits_{i=1}^n \left(q_1(i) - (1-\alpha^*) q_1(i) - \alpha^*\, q_2(i)\right)^2 = \sum\limits_{i=1}^n {\alpha^*}^2 \cdot \left(q_1(i) - q_2(i)\right)^2 
\\ & = {\alpha^*}^2 \cdot \|q_1 - q_2\|_2^2 \leq \|q_1 - q_2\|_2^2 \leq \frac{\epsilon^2}{4 \, n}\,.
\end{align*}
\vspace{-3mm}

\noindent
Hence, $q_1$, which is a mixture distribution with parameter $\alpha = 0$, is a candidate.

We now focus on the case $\|q_1 - q_2\|_2^2 \geq \epsilon^2/(4\, n)$.  Without loss of generality,  $\alpha^* > 1/2$ (otherwise swap $q_1$ and $q_2$).  Fixing $X, Y, Z$, we write (\ref{eq:cls_stat}) as $f(\alpha) = A\alpha^2 + B\alpha + C$, 
where

\begin{equation}\label{eq:A_B_C}
\begin{split}
A  \coloneqq  \sum_{i=1}^n (Y_i - Z_i)^2  - Z_i & - Y_i, \quad
B  := 2 \sum_{i=1}^n Y_i + X_i Y_i +   Y_i Z_i -  Y_i^2-  X_i Z_i, 
\\
& C  \coloneqq \sum_{i=1}^n (X_i - Y_i)^2 - X_i - Y_i.
\end{split}
\end{equation}

As explained earlier, the idea is to use $f(\alpha)$ to find a proper candidate $\alpha$.

We now study the properties of $A$ and $B$.  It turns out that $A$ is the same as the statistic for testing  closeness where $\poi(s)$ samples are drawn from $q_1$ and $q_2$. From \cite{ChanDVV14}, 
$$
\E[X, Y, Z]{A} = s^2 \|q_1 - q_2\|_2^2 \quad \mbox{ and } \quad \Var[X, Y, Z]{A} \leq 8 s^3\, \|q_1 - q_2\|_4^2 \, \sqrt b + 8 \, s^2 \,b\,.
$$
In the following lemma, we show how statistic with similar expected value and variance to $A$ will concenterate:

\begin{restatable}{lemma}{lemLTwoEstimator}\label{lem:l2_estimator}
[Adapted from \cite{ChanDVV14}]
Assume a random variable, namely $R$, has the following properties:
\begin{equation}
    \E{R} = c_1 s^2 \|q_1 - q_2\|_2^2, \quad \quad \Var{R} \leq c_2 s^3 \|q_1 - q_2\|_4^2 \sqrt b + c_3 s^2 b,
\end{equation} where $c_1$, $c_2$, and $c_3$ are three positive constants,  $s$ is an integer,  $q_1$ and $q_1$ are two distributions over $[n]$, and $b$ is a real number which is greater than  $\|q_1\|_2^2$ and $\|q_2\|_2^2$. If $s$ is at least $c \cdot \sqrt b/\tau$ for sufficiently large $c$,
then with probability 0.99 the following is true:
\begin{itemize}
    \item If $\|q_1 - q_2\|_2^2$ is at most $\tau$, then $|R|$ is at most $2\,c_1\,\tau \,s^2$.
    \item If $\|q_1 - q_2\|_2^2$ is at least $\tau$, then $R$ is between $0.9 \cdot \E{R}$ and $1.1 \cdot \E{R}$.
\end{itemize}
\end{restatable}
 
Thus, using the above lemma, we show that 
with probability 0.99, there is a constant $c_A \in [0.9, 1.1]$ such that 
\begin{equation}\label{eq:A}
    A = c_A \cdot s^2 \|q_1 - q_2\|_2^2.
\end{equation} 
$B$ might not have a nice closed-form expression 
when $p$ is an arbitrary distribution, but when $p$ is a mixture, 
it has the following property.
\begin{restatable}{lemma}{lemB}
\label{lem:B}
Suppose $p$ is a mixture of $q_1$ and $q_2$ with parameter $\alpha^* \geq 1/2$. Let $B$ be a function of the sample sets $X$, $Y$, and $Z$ as defined in (\ref{eq:A_B_C}) and let $\gamma < \|q_1-q_2\|_2^2$. If the sample sets, $X$, $Y$, and $Z$, each have $\Theta(\sqrt b/\gamma)$ samples, then with probability 0.99, there exists $c_B \in [0, 1]$ such that
\begin{equation} \label{eq:B}
B = -2 \, c_B \cdot \alpha^* \|q_1 - q_2\|_2^2 \,.
\end{equation}
\end{restatable}
We now analyze $f(\alpha)$ for a fixed $\alpha$.  (The following in fact holds for any distribution $p$ over $[n]$.)
\begin{restatable}{lemma}{lemEVarf}\label{lem:Evarf}
For a fixed $\alpha$,
$$
\E[X, Y, Z] {f(\alpha)} = s^2 \cdot \| p - q_\alpha \|_2^2
\quad \mbox{ and } \quad 
\Var[X, Y, Z] {f(\alpha)} \leq 
8 \, s^3\cdot  \sqrt{b}\cdot \|p - q_\alpha\|_4^2  + 8
\, s^2 \cdot b
\, .
$$
\end{restatable}
By Lemma~\ref{lem:Evarf} and  Lemma~\ref{lem:l2_estimator}, with probability 0.99,  if $p = (1-\alpha^*)q_1 + \alpha^* q_2$, then 
    \begin{equation}\label{eq:f_alpha_star}
    |f(\alpha^*)| \leq s^2 \gamma = T.
\end{equation}
With probability 0.97, all of (\ref{eq:A}), (\ref{eq:B}), and (\ref{eq:f_alpha_star}) hold; we condition on this from now on.

Since $f(\alpha)$ is a quadratic and since $A > 0$ from (\ref{eq:A}), let $\alpha_{\min} = -B/(2A)$ where $f$ achieves its minimum.  We define $\hat \alpha_1$ and $\hat \alpha_2$ as follows:
\begin{equation}
\label{eq:hat_alpha_1_2}
        \hat \alpha_1 = 
        \underset{\alpha \in [\alpha_{\min},1], f(\alpha) \leq T}{\arg\min}
        f(\alpha), 
    \quad \mbox{ and } \quad
    \hat \alpha_2 = \underset{\alpha \in [0, \alpha_{\min}], f(\alpha) \leq T}{\arg\min} f(\alpha) \,.
\end{equation}
Note that (\ref{eq:f_alpha_star}) guarantees that either $\hat \alpha_1$ or $\hat \alpha_2$ exists depending on if $\alpha^* > \alpha_{\min}$ or not; they can also be found very efficiently by binary search.  It remains to show that one of $q_{\hat \alpha_1}$ and $q_{\hat \alpha_2}$ is very close to $p$ in $\ell_2$-distance. 
\begin{restatable}{lemma}{lemGoodAlpha}
\label{lem:good_alpha}
We have either 
$
\|p - q_{\hat\alpha_1}\|_2 \leq \frac{2 \,T}{0.9 \, s^2}$ or $\|p - q_{\hat\alpha_2}\|_2 \leq \frac{2 \,T}{0.9 \, s^2}.$
\end{restatable}
Note that by choice of our parameters, we have $2\,T/(0.9 s^2) < \epsilon^2/(4\,n)$. Hence, either $q_{\hat \alpha_1}$ or $q_{\hat \alpha_2}$ is a candidate.  Thus our potential candidates so far are $\alpha = 0$, $q_{\hat \alpha_1}$, and $q_{\hat \alpha_1}$. In addition, given our assumption for $\alpha^*>1/2$, we need to compute the corresponding ${\hat \alpha_1}$ and ${\hat \alpha_2}$ when $q_1$ and $q_2$ are swapped. Hence, we have at most five candidates for $\alpha$. 
\end{proof}

\subsection{Mixture closeness tester} \label{sec:mix_cls_alg}
In this section, we provide our algorithm and prove its correctness in the following theorem.

\begin{theorem}\label{thm:UB_mixture_closeness} Given a proximity parameter $\epsilon$, 
Algorithm~\ref{alg:closeness} is an closeness tester in the presence of noise that is accessible via samples and it uses $\Theta(\sqrt{n}/\epsilon^2 + n^{2/3}/\epsilon^{4/3})$ samples.
\end{theorem}
\inCOLT{
\begin{wrapfigure}{L}{0.5\textwidth}
\fbox{
\begin{minipage}{0.48\textwidth}
\begin{algorithm}[H]
\DontPrintSemicolon
\SetKwInOut{Input}{Procedure}
\floatconts{alg:closeness}
{\caption{Closeness tester in the  presence \, \,\, \, \, \, of  noise that is accessible via samples. }}
{
{\vspace{3mm}}
{{\bf Procedure:} \textsc{Mixture-Closeness-}} \\
{\hspace*{0.5cm}\textsc{Tester}
($n$, $\epsilon$, sample access to $p$, $q_1$, $q_2$)} \\
    {$k \gets \Theta \left( \min \left(n, {n^{2/3}}/{\epsilon^{4/3}}\right)\right)$ samples to reduce the $\ell_2$-norm of the distribution.} \\
    {$p'$, $q'_1$, $q'_2 \gets $ $p$, $q_1$, and $q_2$ after flattening.} \\
    {$b \gets \Theta\left({1}/{k}\right)$}\\
    {$s \gets \Theta(n \cdot \sqrt b/\epsilon^2)$}\\
    {$X, Y, Z \gets $ $\poi(s)$ samples from each distribution $p'$, $q'_1$, and $q'_2$.}\\
    {$\CC \gets $ set of candidates}\\
    \For{$\alpha \in \CC$}{
        \If {\textsc{$\ell_2^2$-Dist-Estimator}$\left(b, \frac{\epsilon^2}{4\,n}, q'_\alpha, p'\right) \leq \frac{\epsilon^2}{2\,n}$ 
        }
            {\KwRet{\accept}}
    }
    {\KwRet{\reject}}
    {\vspace{3mm}}
}
\end{algorithm}
\end{minipage}
}
\end{wrapfigure}
}{\begin{algorithm}[ht]
\caption{Closeness tester in the presence of  noise that is accessible via samples.}
\label{alg:closeness}
\begin{algorithmic}[1]
\Procedure{Mixture-Closeness--Tester}{$n$, $\epsilon$, sample access to $p$, $q_1$, $q_2$}
    \State{$k \gets \Theta \left( \min \left(n, {n^{2/3}}/{\epsilon^{4/3}}\right)\right)$ samples to reduce the $\ell_2$-norm of the distribution.}
    \State{$p'$, $q'_1$, $q'_2 \gets $ $p$, $q_1$, and $q_2$ after flattening.}
    \State{$b \gets \Theta\left({1}/{k}\right)$}
    \State{$s \gets \Theta(n \cdot \sqrt b/\epsilon^2)$}
    \State{$X, Y, Z \gets $ $\poi(s)$ samples from each distribution $p'$, $q'_1$, and $q'_2$.}
    \State{$\CC \gets $ set of candidates}
    \For{$\alpha \in \CC$}{
        \If {\textsc{$\ell_2^2$-Dist-Estimator}$\left(b, \frac{\epsilon^2}{4\,n}, q'_\alpha, p'\right) \leq \frac{\epsilon^2}{2\,n}$ 
        }
            \State{\Return{\accept}}
        \EndIf
    }
    \EndFor
    \State{\Return{\reject}}
\EndProcedure
\end{algorithmic}
\end{algorithm}
}
\begin{proof}\;  We reduce the $\ell_2$-norm of the three input distributions via the 
reshaping technique proposed in \cite{DiakonikolasK:2016}. 
Let $S$ be a multi-set consisting of $3\,k$ samples, where $k$ samples are chosen from each distribution $p$, $q_1$, and $q_2$. For $i \in [n]$, we assign $b_i$ buckets to element $i$ where $b_i$ is the number of instances of element $i$ in set $S$ plus one. For a  distribution $d$ over $[n]$, we define $d'$ to be a distribution over all the buckets, $D \coloneqq \{(i,j)~\mid~i \in [n] \mbox{ and } j \in [b_i]\}$. We generate a sample from $d'$ via the following process: (i) draw a sample $i \sim d$, (ii) pick $j \in [b_i]$ uniformly at random, and (iii) output $(i,j)$. The probability of any element $(i,j)$ according to $d'$ is $d(i)/b_i$. It is known that flattening does not change the $\ell_1$-distance between two distributions.  Let $p'$, $q'_1$, and $q'_2$ be the distributions $p$, $q_1$, and $q_2$ after flattening. We show that a mixture distribution will remain a mixture after flattening. More precisely, if $p$ is a mixture of $q_1$ and $q_2$ with parameter  $\alpha^*$, then it is easy to see that $p'$ is a mixture of distributions $q'_1$ and $q'_2$ with the same parameter $\alpha^*$.
Thus, it suffices to test if $p'$ is a mixture of $q'_1$ and $q'_2$.

By setting $k = \Theta(\min(n, n^{2/3}/\epsilon^{4/3}))$, according to \cite[Lemma II.6]{DiakonikolasK:2016} and Markov's inequality, we can assume the $\ell_2$-norms of all three distributions $p'$, $q'_1$, and $q'_2$ are at most $b$ with probability at least $0.99$, where we set $b = 1/\min(n, n^{2/3}/\epsilon^{4/3}) = \Theta(1/k)$ . Also, note tht $|D| = \Theta(n)$. 

Given Theorem~\ref{thm:proper_candidates}, one can find a set $\MM$ of at most five candidates. If $p'$ is a mixture of $q'_1$ and $q'_2$, then there is an $\alpha \in \MM$ such that $\| q'_{\alpha} - p' \|_2 \leq \epsilon/(2\sqrt{|D|})$.  On the other hand, if $p'$ is $\epsilon$-far from being a mixture, it is also $\epsilon$-far from all $\alpha \in \MM$; using the Cauchy--Schwarz inequality, we have $\|q'_{\alpha} - p'|_2 \geq \epsilon/\sqrt{|D|}$. Note that \cite{ChanDVV14} showed one can estimate the $\ell_2$-distance accurately using $\Theta(|D| \cdot \sqrt b/\epsilon^2)$ samples and with probability 0.99 (see Lemma~\ref{lem:l22} in Section~\ref{sec:closeness_lemmas}.)  

By a union bound, the probability that $\MM$ does not contain the right $\alpha$, the probability that the $\ell_2$ estimation fails, and the probability that $\poi(\lambda) < 100 \lambda$ sum up to below $1/3$.  
Hence, with probability 2/3, the algorithm outputs the right answer and the total number of samples is $\Theta(k + n \sqrt{b}/\epsilon^2) = \Theta(\sqrt{n}/\epsilon^2 + n^{2/3}/\epsilon^{4/3})$. 
\end{proof}

\subsection{Proofs for Section~\ref{sec:prp_candidates} and Section~\ref{sec:mix_cls_alg}} \label{sec:closeness_lemmas}

In this section, we present the proofs of the lemmas presented earlier in Section~\ref{sec:prp_candidates} and Section~\ref{sec:mix_cls_alg}.

\lemLTwoEstimator*
\begin{proof}
We use Chebyshev's inequality to prove the lemma. For the first case, by the $\ell_p$-norms inequality, we have the following:
\begin{align*}
    \Pr{|R - \E{R}| \geq c_1\,\tau s^2} & \leq \frac{\Var{R}}{c_1^2\tau^2 s^4} \leq \frac{c_2 \|q_1 - q_2\|_4^2 \sqrt b}{c_1^2 \tau^2 s} + \frac{c_3 \, b}{c_1^2 \tau^2 s^2} 
\\
    & \leq \frac{c_2 \|q_1 - q_2\|_2^2 \sqrt b}{c_1^2 \tau^2 s} + \frac{c_3 \, b}{c_1^2 \tau^2 s^2} \leq \frac{c_2 \sqrt b}{c_1^2 \tau\, s} + \frac{c_3 \, b}{c_1^2 \tau^2 s^2} \leq 0.01,
\end{align*}
where the last inequality is true when $s \geq  \max(200 \, c_2/c_1^2, \sqrt{200 \, c_3}/c_1) \sqrt  b/\tau$.

For the second case, we have the following:
\begin{align*}
    \Pr{|R - \E{R}| \geq 0.1 \cdot \E{R}} & \leq \frac{100 \, \Var{R}}{\E{R}^2} \leq \frac{100 \,c_2 \|q_1 - q_2\|_4^2 \sqrt b}{c_1^2 \, s \,\|q_1 - q_2\|_2^4  } + \frac{100 \,c_3 \, b}{c_1^2 \, s^2 \,\|q_1 - q_2\|_2^4  }
\\
    & \leq \frac{100 \,c_2 \sqrt b}{c_1^2 \, s \,\|q_1 - q_2\|_2^2  } + \frac{100 \,c_3 \, b}{c_1^2 \, s^2 \,\|q_1 - q_2\|_2^4  }
\\ 
    & \leq \frac{100 \,c_2 \sqrt b}{c_1^2 \tau\, s} + \frac{100 \,c_3 \, b}{c_1^2 \tau^2 s^2}
    \leq 0.01,
\end{align*}
where the last inequality is true when $s \geq  \max(20000 \, c_2/c_1^2, \sqrt{20000 \, c_3}/c_1) \sqrt  b/\tau$.  This completes the proof.
\end{proof}

\lemB*
\begin{proof}
Recall that $B$ is defined to be $2 \sum_{i=1}^n Y_i + X_i Y_i +   Y_i Z_i -  Y_i^2-  X_i Z_i$. To analyze the expected value and the variance of $B$, we consider each terms in the sum. Let $B_i$ denote a single term in the sum after ignoring constant 2:
$$
B_i \coloneqq Y_i + X_i Y_i +   Y_i Z_i -  Y_i^2-  X_i Z_i \,.
$$
Note that via the Poissonization method, the $X_i$'s, the $Y_i$'s, the $Z_i$'s, and consequently the $B_i$'s are independent random variables. Note that if $x$ is a Poisson random variable with mean $\lambda$, then 
$\E{x^2}$ is $\lambda^2 + \lambda$. 
Using this equation, we compute the expected value of $B_i$:
\begin{align*}
\E[X, Y, Z] {B_i} & = - s^2 \left( q_1(i)^2 + p(i) q_1(i) + q_1(i) q_2(i) - p(i) q_2(i)\right)
\\ & = - \alpha^
* s^2 (q_1(i) - q_2(i))^2 \,.
\end{align*}
Thus, the expected value of $-B$ is the following:
$$\E[X, Y, Z]{-B}= - \sum\limits_{i=1}^n 2 B_i= \sum\limits_{i=1}^n 2\alpha^* s^2 (q_1(i) - q_2(i))^2 = 2 \alpha^* s^2 \|q_1 - q_2\|_2^2,$$
where $2\alpha^*$ is a constant between $[1, 2]$. Using the first four moments of the Poisson distribution and the fact that $\alpha \leq 1$, we have the following: 
\begin{align*}
\Var[X,Y,Z]{B_i} & = \E[X,Y,Z]{B_i^2} - \E[X,Y,Z]{B_i}^2
\\ & = 
s^3 \, \alpha^* ( 1 + \alpha^*) \, (q_1(i)-q_2(i))^2 (q_1(i)+q_2(i)) +
2 \, s^3 \, (q_1(i)-q_2(i))^2 \, q_1(i) 
\\ & + 
s^2  \left(\alpha^* (q_2(i)^2 - q_1(i)^2)  + q_1(i) \, (3q_1(i) + 2q_2(i))\right)
\\ & \leq
4 \, s^3 \, (q_1(i) + q_2(i)) (q_1(i)- q_2(i))^2 + s^2 \, (q_1(i) + q_2(i))^2  + 3\, s^2 \, q_1(i)^2
\,.
\end{align*}
Using the bound above and the Cauchy--Schwarz inequality, we bound the variance of $B$ as follows:
\begin{align*}
\Var[X,Y,Z]{B} & = 4 \sum_{i=1}^n \Var[X,Y,Z]{B_i} 
\\ & \leq 
16 \, s^3 \, \sum_{i=1}^n (q_1(i) + q_2(i)) (q_1(i)- q_2(i))^2 + 8 \|q_2\|_2^2 + 20 \, s^2 \|q_1\|_2^2
\\ & \leq  
16 s^3 \sqrt{\left(\sum_{i=1}^n  (q_1(i) - q_2(i))^4 \right) \cdot  \left(\sum_{i=1}^n (q_1(i) + q_2(i))^2 \right)} + 28 \,s^2 b
\\ & \leq 
32 \, s^3 \|q_1-q_2\|_4^2 \sqrt{b} + 28
\, s^2 \, b
\,.
\end{align*}
Clearly, the variance of $-B$ is equal to the variance of $B$, and it is bounded the same as above. Note that $\gamma $ is at most $\|q_1-q_2\|_2^2$, and the sample sets, $X$, $Y$, and $Z$, each have at least $\Theta(\sqrt b/\gamma)$ samples.
By Lemma~\ref{lem:l2_estimator}, 
there exists $c_B \in [0, 1]$ such that
$$
- B = 2 \, c_B \, \alpha^* \|q_1 - q_2\|_2^2, 
$$
with probability 0.99 which concludes the proof. 
\end{proof}

\lemEVarf*
\begin{proof} In this proof, we adapt the proof of Proposition 3.1 from \cite{ChanDVV14}. Recall that 
$$f(\alpha) = \sum\limits_{i=1}^n (X_i - (1-\alpha)Y_i - \alpha Z_i)^2 - X_i - (1-\alpha)^2 Y_i - {\alpha}^2 Z_i.$$
Via the Poissonization method, we can assume $X_i$ (similarly $Y_i$ and $Z_i$) is a random variable from $\Poi(s \, p(i))$ (similarly $\Poi(s \, q_1(i))$ and $\Poi(s \, q_2(i))$), which is drawn independently from the rest of the random variables. Note that if $x$ is a Poisson random variable with mean $\lambda$, then 
$\E{x^2}$ is $\lambda^2 + \lambda$. 
Using this equation and the independence of the random variables, for a fixed $\alpha$, we have:
$$
\E[X, Y, Z] {f(\alpha)} = s^2 \cdot \| p - q_\alpha \|_2^2 \, .
$$

Now, we bound the variance of $f(X,Y,Z, \alpha)$ for a fixed $\alpha$.
Let $W_i$ denote a single term in the summation:
$$
W_i \coloneqq \left(X_i - (1-\alpha)Y_i - \alpha Z_i\right)^2 - X_i - (1-\alpha)^2 Y_i - \alpha^2 Z_i \, .$$
Using the moments of the Poisson distribution, we have 
\begin{align*}
\Var[X, Y, Z]{W_i} & = \E[X, Y, Z]{W_i^2} - \E[X, Y, Z]{W_i}^2 
\\ & = 
4 \, s^3\, (p(i) - (1 - \alpha) q_1(i) - \alpha \, q_2(i))^2\cdot (p(i) + (1 - \alpha)^2 q_1(i) + \alpha^2 q_2(i)) 
\\ & + 2\, s^2\, (p(i) + (1 - \alpha)^2 q_1(i) + \alpha^2 r
(i))^2
\\ & \leq 
4 \, s^3\, (p(i) - (1 - \alpha) q_1(i) - \alpha \, q_2(i))^2\cdot (p(i) + (1 - \alpha) q_1(i) + \alpha q_2(i)) 
\\ & + 2\, s^2\, (p(i) + (1 - \alpha) q_1(i) + \alpha q_2(i))^2
\\ &= 
4 \, s^3\, (p(i) - q_{\alpha}(i))^2\cdot (p(i) + q_{\alpha}(i)) 
 + 2\, s^2\, (p(i) + q_{\alpha}(i))^2.
\end{align*}
Now, we bound the variance of $f$ which is the sum of $n$ independent terms, $W_i$'s. Using the Cauchy--Schwarz inequality, and the fact that $(p(i) + q_{\alpha}(i))^2$ is at most $2 p(i)^2 + 2 q_{\alpha}(i)^2$, we have
\begin{align*}
\Var[X, Y, Z] {f(\alpha)} & = \sum_{i=1}^n \Var[X, Y, Z] {W_i}
\\ & \leq  
4 \, s^3\, (p(i) - q_{\alpha}(i))^2\cdot (p(i) + q_{\alpha}(i)) 
 + 2\, s^2\, (p(i) + q_{\alpha}(i))^2 
\\ & \leq 
4 \, s^3 \, \sqrt{\left(\sum_{i=1}^n  (p(i) - q_{\alpha}(i))^4 \right) \cdot  \left(\sum_{i=1}^n (p(i) + q_{\alpha}(i))^2 \right)}
+ 4 \, s^2 \left(\|p
\|_2^2 + \|q_\alpha\|_2^2\right)
\\ & \leq 
8 \, s^3 \cdot \|p - q_\alpha\|_4^2 \cdot \sqrt{b} + 8
\, s^2 \, b \,,
\end{align*}
where $b$ is at least $\|p\|_2^2$ and $\|q_\alpha\|_2^2$ by the first condition of the theorem. 
\end{proof}

\lemGoodAlpha*
\begin{proof}
Consider the statistic as a function of $\alpha$: $f(\alpha) = A \alpha^2 + B \alpha + C$. Since $A$ is positive, $f$ takes its minimum at $\alpha_{\min} \coloneqq (-B)/2A$. By Equation~\ref {eq:A} and Equation~\ref{eq:B},  $\alpha_{\min}$ is  $c_B \alpha^*/c_A$, and for any $\alpha$, we have:
\begin{equation} \label{eq:delta_f}
\begin{split}
f(\alpha^*) - f(\alpha) & = A ({\alpha^*}^2 - \alpha^2) + B ({\alpha^*} - \alpha)
\\ & = c_A \, s^2 \, \|q_1 - q_2\|_2^2 \, ({\alpha^*} - \alpha) \left(
{\alpha^*} + \alpha - \frac{2 \alpha^* (c_B)}{c_A}
\right)
\\ & = c_A \, s^2 \, \|q_1 - q_2\|_2^2 \, ({\alpha^*} - \alpha) \left(
{\alpha^*} + \alpha -2 \alpha_{\min}
\right).
\end{split}
\end{equation}

Depending on whether $\alpha^*$ is larger than  $\alpha_{\min}$ or not, we consider the following cases.

\noindent
\textit{Case 1: $\pmb {\alpha^* \geq \alpha_{\min}}$.} Let $\hat \alpha_1$ be the smallest number in $[\alpha_{\min}, 1]$ for which $|f(\hat \alpha_1 )|$ is at most $T$. Clearly, $\hat \alpha_1$ exists since $\alpha^*$ is a potential solution, so the solution interval is not empty. Note that based on the way we pick $\hat \alpha_1$, the following are true: (i) $\hat \alpha_1$ is at most $\alpha^*$, (ii) $f(\hat \alpha_1)$ is at least $- T$, and (ii) since $A$ is positive, and $f$ is increasing over $[\alpha_min, 1]$, then $f(\hat \alpha_1)$ is at most $f(\alpha^*)$. Hence, by Equation~\ref{eq:delta_f}, we have:
$$0 \leq f(\alpha^*) - f(\hat \alpha_1) = c_A \, s^2 \, \|q_1 - q_2\|_2^2 \, ({\alpha^*} - \hat \alpha_1) \left(
{\alpha^*} + \hat \alpha_1 -2 \alpha_{\min} \right) \leq 2 \, T \,.$$
If we replace ${\alpha^*} + \hat \alpha_1 - 2\alpha_{\min}$ by a smaller quantity ,$\alpha^* - \hat \alpha_1$, where both are positive then we have:
\begin{equation}\label{eq:cls_bounding_l2_case1}
s^2 \, \|q_1 - q_2\|_2^2 \, ({\alpha^*} - \hat \alpha_1)^2 \leq \frac{2 \, T}{c_A} \leq \frac 2{0.9} \, T  \,.
\end{equation}

\noindent
\textit{Case 2: $\pmb {\alpha^* \leq \alpha_{\min}}$.} We  replicate what we did in the previous case.  Let $\hat \alpha_2$ be the largest number in $[0, \alpha_{\min}]$ for which $|f(\hat \alpha_2 )|$ is at most $T$. Clearly, $\hat \alpha_2$ exists since $\alpha^*$ is a potential solution, so the solution interval is not empty. Note that based on the way we pick $\hat \alpha_2$, the following are true: (i) $\hat \alpha_2$ is at least $\alpha^*$, (ii) $f(\hat \alpha_2)$ is at least $- T$, and (iii) since $A$ is positive, and $f$ is decreasing  over $[0, \alpha_min]$, then $f(\hat \alpha_2)$ is at most $f(\alpha^*)$. Hence, by Equation~\ref{eq:delta_f}, we have:
\begin{align*}
    0 \leq f(\alpha^*) - f(\hat \alpha_2) & = c_A \, s^2 \, \|q_1 - q_2\|_2^2 \, ({\alpha^*} - \hat \alpha_2) \left(
{\alpha^*} + \hat \alpha_2 -2\alpha_{\min} \right)
\\ & = c_A \, s^2 \, \|q_1 - q_2\|_2^2 \, (\hat \alpha_2 - {\alpha^*}) \left(2\alpha_{\min} - {\alpha^*} - \hat \alpha_2 \right) \leq 2 \, T \,.
\end{align*}
If we replace $2\alpha_{\min}{\alpha^*}-\hat \alpha_2 $ by a smaller quantity, $\hat \alpha_2 - \alpha^* $, where both are positive, then we have:
\begin{equation}\label{eq:cls_bounding_l2_case2}
s^2 \, \|q_1 - q_2\|_2^2 \, ({\alpha^*} - \hat \alpha_2)^2 \leq \frac{2 \, T}{c_A} \leq \frac 2{0.9} \, T  \,.
\end{equation}

The left side of  Equation~\ref{eq:cls_bounding_l2_case1} and Equation~\ref{eq:cls_bounding_l2_case2} are in the form of  the $\ell_2$-distance between two mixture distributions $p$ and $q_{\hat \alpha}$ due to the following:
\begin{align} \label{eq:mixture-l2-norm}
\|p - q_{\hat \alpha}\|_2^2 & = \sum_{i=1}^n \left(p(i) - q_{\hat\alpha}(i)\right)^2 = \sum_{i=1}^n \left((1 - \alpha^*) \,q_1(i) + \alpha^* q_2(i) - (1-\hat \alpha) \, q_1(i) - \hat \alpha\, q_2(i) \right)^2
\\ & =  (\alpha^* - \hat \alpha)^2 \sum_{i=1}^n \left(q_1(i) - q_2(i) \right)^2 = (\alpha^* - \hat \alpha)^2 \|q_1 - q_2\|_2^2
\,.
\end{align}
Note that we are either in case 1 or case 2. So, on of the two equations, Equation~\ref{eq:cls_bounding_l2_case1}, Equation~\ref{eq:cls_bounding_l2_case2} has to be true. By Equation~\ref{eq:mixture-l2-norm}, of the following is true. 
$$
\|p - q_{\hat\alpha_1}\|_2 \leq \frac{2 \,T}{0.9 \, s^2} \quad \quad \quad \|p - q_{\hat\alpha_2}\|_2 \leq \frac{2 \,T}{0.9 \, s^2},
$$
which concludes the proof.
\end{proof}

\inCOLT{
\begin{algorithm}[t]
\DontPrintSemicolon
\SetKwInOut{Input}{Procedure}
\floatconts{alg:l2_estimator}
{\caption{An  algorithm for estimating the $\ell_2$-distance squared \cite{ChanDVV14}}}
{\Input{\textsc{$\ell_2^2$-Estimator}{($b, \sigma$, sample access to $r_1$ and $r_2$)}}
    {$s \gets \Theta(\sqrt b/\sigma)$)}\\
    {$X \gets $ Draw $s$ samples from $r_1$.}\\
    {$Y \gets $ Draw $s$ samples from $r_2$.}\\
    {\KwRet{$\frac{1}{s^2}\sum\limits_{i=1}^n (X_i - Y_i)^2 - X_i - Y_i$}}
}
\end{algorithm}
}
{
\begin{algorithm}[t]\label{alg:l2_estimator}
\caption{An  algorithm for estimating the $\ell_2$-distance squared \cite{ChanDVV14}}
\begin{algorithmic}[1]
\Procedure{$\ell_2^2$-Estimator}{$b, \sigma$, sample access to $r_1$ and $r_2$}
    \State {$s \gets \Theta(\sqrt b/\sigma)$)}
    \State {$X \gets $ Draw $s$ samples from $r_1$.}
    \State {$Y \gets $ Draw $s$ samples from $r_2$.}
    \State {\Return {$\frac{1}{s^2}\sum\limits_{i=1}^n (X_i - Y_i)^2 - X_i - Y_i$}}
\EndProcedure
\end{algorithmic}
\end{algorithm}
}

\begin{lemma}\label{lem:l22}[Restated from \cite{ChanDVV14}]
The procedure \textsc{$\ell_2^2$-Estimator $(b, \sigma, r_1, r_2)$} described in Algorithm~\ref{alg:l2_estimator}, that uses $\Theta(\sqrt b/\sigma)$ samples, has the following property with probability 0.99:
\begin{itemize}
    \item If $\|r_1 - r_2\|_2^2$ is at most $\sigma$, then $|R|$ is at most $2\,\sigma \,s^2$.
    \item If $\|r_1 - r_2\|_2^2$ is at least $\sigma$, then $R$ is between $0.9 \cdot \E{R}$ and $1.1 \cdot \E{R}$.
\end{itemize}
\end{lemma}

\begin{proof}\; 
We use the $\ell_2^2$-distance estimator proposed in \cite{ChanDVV14}. However, for the sake of completeness, we provide the process in Algorithm~\ref{alg:l2_estimator}. $X$ and $Y$ are two sample sets each containing $s$ samples from $r_1$ and $r_2$ respectively. 
Let $X_i$ and $Y_i$ indicate the numbers of samples in $X$ and $Y$ respectively. The authors showed that the expected value of the statistic $\sum_{i=1}^n (X_i - Y_i)^2 - X_i - Y_i$ is $s^2\|r_1 - r_2\|_2^2$, and the variance is bounded by $8 s^3\, \|r_1 - r_2\|_4^2 \, \sqrt b + 8 \, s^2 \,b$.  By Lemma~\ref{lem:l2_estimator}, if we draw $\Theta(\sqrt{b}/\gamma)$ samples, then the algorithm will have the desired property with probability 0.99.
\end{proof}

\section{Testing under \texorpdfstring{$k$}{k}-flat noise}\label{sec:kFlat}

We have so far considered the problems of identity testing and closeness testing in the presence of the noise that is directly accessible and proved these problems have the same sample complexity as their respective noise-free versions.  These results raise the question of whether one can replace the requirement of access to the noise by an assumption that restricts the noise to be in a {\em class} of distributions and still achieve improved sample complexity compared to the near-linear lower bound we mentioned earlier.  In this section we develop a tester for identity testing when the noise distribution belongs to the class of {\em $k$-flat distributions} without any further information. This assumption means that the noise can be {\em any} $k$-flat distribution, while the parameters of the $k$-flat distribution are not known to the tester, nor given
via samples. 

\subsection{Preliminaries}
We begin by formally defining $k$-flat distributions:
We say $\II = \{I_1, \ldots, I_k\}$ is a {\em $k$-segmentation} of $[n]$ if and only if  $I_1, \ldots, I_k$ are $k$ disjoint intervals that cover $[n]$. Also, we say a function $f:[n] \rightarrow \mathbb R $ is a {\em $k$-flat function} if and only if there is a $k$-segmentation of $[n]$, namely $\II = \{I_1, \ldots, I_k\}$, such that 
for any two elements, $x$ and $y$, in the same interval in $\II$, $f(x)$ is equal to $f(y)$. A distribution is a {\em $k$-flat distribution} if and only if its probability mass function is a $k$-flat function. 

We next define concepts that will be necessary for describing our algorithms.
For any distribution $p$ and a partition $\DD = \{ D_1, \ldots D_t \}$ of its domain, the \emph{coarsening} of $p$ over $\DD$, denoted by $p_{\tuple{D}}$, is a distribution over the sets in $\DD$ where the probability  of each set $D_i$ is $\sum_{x \in D_i} p(x)$.  For a subset $D \subseteq [n]$, we define the \emph{restriction} of $p$ to $D$, denoted by $p_{|D}$, to be a distribution over $D$ for which the probability of $x \in D$ is equal to $p(x~\mid~x \in D)$. Although the restriction is well-defined only when $p(D)$ is not zero, abusing notation, we define $\| p_{|D} - q_{|D} \|_1$ to be zero if $p(D)$ or $q(D)$ is zero.

Also, throughout this section, we study different schemes for partitioning the domain. In addition to $k$-segmentation, which is
defined earlier, two other schemes are defined as follows: Given a known distribution $q$, Batu et al. in~\cite{BatuFFKRW} 
provide a partitioning scheme, called {\em bucketing}, which places elements with similar probability 
in the same bucket. Note that, in contrast with $k$-segmentation, 
this scheme does not necessarily place consecutive elements in the same bucket. 
\begin{defn} [Similar to \cite{BatuFFKRW}] \label{def:bucketing}
Assume we have a known distribution $q$ over $[n]$. Given a parameter $\epsilon$, we define the {\em bucketing of the domain}, \textsc{Bucket}$(q, n, \epsilon)$, to be a set of $\vv$ subsets of the domain, $\BB = \{B_1, \ldots, B_\vv\}$, where each subset is defined as below:
$$B_1 \defeq \left\{x \in [n]\left| q(x) \leq \frac {\epsilon^2} n\right.\right\}, \mbox{ and}$$
$$ B_i \defeq \left\{x \in [n] \left | \frac{(1+\epsilon)^{i} \epsilon^2 }{n} < q(x) \leq \frac{(1+\epsilon)^{i+1} \epsilon^2}{n} \right.\right\} \quad \mbox{for } i = 2, \ldots, \vv\,.$$
\end{defn}
We define the last partitioning scheme below. This partition is a refinement of the bucketing with respect to a $k$-segmentation $\II$.

\begin{defn} \label{def:division}
Assume $\II$ is a $k$-segmentation of $[n]$, and $\BB$ is a bucketing of $[n]$ containing $\vv$ disjoint subsets. We define $\DD(\II, \BB) = \{D_{i,j, \ell}\}_{(i,j) \in [k]\times[\vv]}$ to be {\em a division of the domain} for which the $D_{i,j, \ell}$'s are the intersection of the $i$th interval and the $j$th bucket. Formally, $D_{i,j, \ell}$ is defined as:
$$D_{i,j, \ell} \coloneqq  \{x \in [n] \,|\,x \in I_i \cap B_j\} \,.$$
\end{defn}

\paragraph{The problem of testing identity in the presence of $k$-flat noise.} Suppose we are given a known distribution $q$, and sample access to a distribution $p$ both over the domain $[n]$. Let $\CC$ denote the class of all $k$-flat distributions over $[n]$.  The problem of testing identity in the presence of $k$-flat noise boils down to distinguishing the following cases with probability at least 2/3:
\begin{itemize}
    \item  There exists a mixture parameter $\alpha^*$ and a 
$k$-flat distribution $r^*$ over $[n]$
such that $p$ is a mixture of $q$ and $r^*$ with parameter $\alpha^*$, i.e., $p = (1-\alpha^*) q + \alpha^* \, r^*$.
    \item $p$ is $\epsilon$-far from any distribution of the form $(1-\alpha) q + \alpha \, r$ where $r \in \CC$ and $\alpha \in [0,1]$.
\end{itemize}

\subsection{The algorithm} \label{sec:kFlatAlg}
We start by explaining the properties of the partitioning schemes we defined earlier. Let $\BB = \textsc{Bucket}(q, n, \epsilon')$ be the bucketing of the domain elements for a parameter $\epsilon' \coloneqq \epsilon/14$. The algorithm can obtain this bucketing since $q$ and $\epsilon$ is known to the algorithm.
The bucketing scheme is designed such that the probabilities of the elements in a bucket are within a $(1+\epsilon)$-factor of each other (except for $B_1$). This property implies that the restriction of $q$ to any bucket is extremely close to the uniform distribution. 

Now, assume that $p$ is in fact a mixture of $q$ and a $k$-flat distribution $r^*$. We denote the $k$-segmentation of $r^*$ by $\II^* = \{I_1^*, \ldots, I_k^*\}$ (which is not known to the algorithm). By definition, the restriction of $r^*$ on any $I^*_i \in \II^*$  is a uniform distribution. Consider the division $\DD(\II^*, \BB)$, described in
Definition \ref{def:division}. Observe that $D_{i,j, \ell} \in \DD$ is a subset of both $I_i$ and $B_j$. One can show that the restriction of $r^*$ is uniform on $D_{i,j, \ell}$, and the restriction of $q$ to $D_{i,j, \ell}$ is very close to the uniform distribution as well. Thus, $p$, which is assumed to be the mixture of $q$ and $r^*$, must be very close to the uniform distribution on $D_{i,j, \ell}$. We formally prove this claim in Lemma~\ref{lem:smooth_Mix}.

Based on the above observation, our tester looks for two qualities in $p$ to assert that it is a mixture distribution: Given a division $\DD(\II,\BB)$, (i) are the restrictions of $p$ to the $D_{i,j, \ell}$'s almost uniform and (ii) is the overall shape of $p$ over $D_{i,j, \ell}$'s (i.e.,  $p_{\tuple{\DD}}$) consistent with
a mixture of $q$ and a $k$-flat noise distribution?
More specifically, our tester follows these steps.
For every $k$-segmentation $\II$, the tester 
checks  that the restriction of $p$ to 
each $D_{i,j, \ell} \in \DD(\II, \BB)$ is almost uniform. If it figures out that it is not the case, it abandons the current segmentation, and start over with another one. If at some point, the tester passes this step, it checks the overall shape of $p$. It draws enough samples from $p$ and forms the empirical distribution $\hat p$ from the samples. Then it checks whether there exists a $k$-flat {\em function}, 
$f$, such that $\hat p_{\tuple{\DD}}$  is consistent with 
a mixture of $q$ and $f$. 
If the tester finds a $k$-segmentation such that the distribution passes the two steps above, then it  asserts  that $p$ is a mixture and outputs \accept. Otherwise, it outputs \reject. 

Based on our first observation, one can expect the tester to accept a mixture distribution $p$. However, the main challenge is to show that the tester rejects when $p$ is $\epsilon$-far from being a mixture. To prove this fact, we also use the following observation. Suppose we have two distributions $p$ and $p'$. Let $\mathcal P$ be a partition of their domain. We prove that if $p$ and $p'$ are $\epsilon$-far from each other, there is a noticeable discrepancy between either their coarsening distributions over $\mathcal P$ or their restrictions to the subsets in $\mathcal P$ (Lemma~\ref{lem:distL1ToCoarsening}). This observation implies that if $p$ is $\epsilon$-far from being a mixture distribution, then at least one the steps will fail. Hence, we distinguish both cases with high probability.   

We describe our tester in Algorithm~\ref{alg:kFlat-known-I} and show its correctness in Theorem~\ref{thm:kflat-known_I}. 
Later, we also discuss how to avoid trying all $\II$'s and achieve a polynomial time algorithm.  All missing proofs in the rest of this section are in Section~\ref{sec:kFlatProofs}.

\begin{algorithm}[ht]
\caption{Identity testing in the presence of $k$-flat noise}
\label{alg:kFlat-known-I}
\begin{algorithmic}[1]
\Procedure{Identity-Tester-k-Flat-Noise}{$q, n, \epsilon$, sample access to $p$}
    \State{$\epsilon' \gets \epsilon/14$}
    \State{$\BB \gets $ \textsc{$(q, n,, \epsilon')$}}
    \State{$M \gets $ Multiset of $s = \widetilde{\Theta}_\epsilon(k \sqrt{n})$ i.i.d. samples from $p$.}
    \For {every possible $k$-segmentation $\II$}
        \State {$\DD \gets \DD(\II, \BB)$}
        \For{$i = 1$ to $k$}
            \For{$j = 2$ to $\vv$}
                \State {$M_{i,j} \gets M \cap D_{i,j, \ell}$}
                \If{$|M_{i,j}| \geq \epsilon' s /4(k\cdot\vv)$}
                    \If { $\|p_{|D_{i,j, \ell}} - \UU_{|D_{i,j, \ell}} \|_2 \geq 2 \, {\epsilon'}/\sqrt{|D_{i,j, \ell}|}$}
                        \State{Continue with another segmentation.   \label{line:kFlatRjctInD}}
                    \EndIf
                \EndIf
            \EndFor
        \EndFor
        \State {$\hat p \gets$ Empirical distribution built by samples in $M$}
        \For{$\alpha = 0, \epsilon'/2, \epsilon', \ldots, 1$}
            \State{$f \gets$ find a $k$-flat function on $\II$ such that $\hat p_{\tuple{\DD}}$ is $2\,\epsilon'$-close to $(1-\alpha)q_{\tuple{\DD}} + \alpha \,f_{\tuple{\DD}}$}
            \If{$f$ exists }
                \State{\Return {\accept} \label{line:accKflatFunc}}
            \EndIf
        \EndFor
    \EndFor
    \State{\Return {\reject}}
\EndProcedure
\end{algorithmic}
\end{algorithm}

\begin{theorem} \label{thm:kflat-known_I} Algorithm~\ref{alg:kFlat-known-I} is an identity tester in the presence of  $k$-flat noise that uses $\widetilde{O}(\sqrt{n k}/\epsilon^{3.5})$ samples.
\end{theorem}
\begin{proof}\;We set $\epsilon' = \epsilon/14$. We denote the number of buckets in $\BB = \textsc{Buckets}(q, n, \epsilon')$ by $\vv$. Let $t$ denote $k \cdot \vv$.
Without loss of generality we assume $t \leq n$. Otherwise, one could learn the distribution $p$ up to $\epsilon/2$ $\ell_1$-distance error via $O(n/\epsilon^2)$ samples, and trivially check if it is $\epsilon/2$-close to a mixture of $q$ and a $k$-flat distribution.

Consider a segmentation $\II$, and a division $\DD \coloneqq  \DD(\II, \BB)$. 
To obtain better sample complexity, we need to make sure that the size of each set in $\DD$ is not greater than $\ceil{n/t}$. In the case that a large set of size $z > \ceil{n/t}$ exists, we split it into $D_{i,j, \ell} \coloneqq \floor{z\cdot t/ n} + 1$ sets of roughly the same size and denote them by $D_{i,j, \ell}$ for $\ell \in [D_{i,j, \ell}]$. The new sets form a new partition of the domain. We call it a {\em refined division}, denoted $\widetilde{\DD} \coloneqq \widetilde{\DD}(\II, \BB, n, t)$. Note that this replacement will not asymptotically increase the total number of sets in the division, since $\widetilde{\DD}$ has $\sum_{i,j} D_{i,j, \ell} \leq 2\,t$ many sets. 

Now, we establish that for a sufficiently large number of samples, the three steps in the algorithm succeed with high probability.  First, in the following lemma, we show that $O(t \cdot \log n/\epsilon^2)$ samples are enough to obtain an empirical distribution $\hat p$ such that for all the divisions $\DD$ $\hat p_{\tuple{\DD}}$ and $p_{\tuple{\DD}}$ are $\epsilon'$-close to each other with probability 0.9. 

\begin{restatable}{lemma}{lemMultiEmpirical}\label{lem:MultiEmpirical}
Assume $p$ is a distribution over $[n]$. Let $\hat p$ be an empirical distribution formed by $\Theta(\min(n, k \vv \log n)\cdot (\log \delta^{-1})/\epsilon'^2)$ samples from $p$. 
Fix a bucketing of the domain $\BB = $\textsc{Bucket}$(q, n, \epsilon')$. 
For every $k$-segmentation $\II$, and the corresponding refined division of the domain $\DD = \DD(\II, \BB)$, the coarsening of $p$ and the empirical distribution $\hat p$ over $\DD$ is at most $\epsilon'$-far from each other with probability at least $1-\delta$.
\end{restatable}

Second, we show that if $p(D_{i,j, \ell})$, for a fixed $i,j,$ and $\ell$, is at least $\epsilon'/|\widetilde{\DD}| = \Theta(\epsilon'/t)$, 
then $M_{i,j}$ contains at least $\epsilon'/(4t)$ fraction of the samples with high probability.
Note that there are at most $\Theta(n^2 \cdot \vv)$ set $D_{i,j, \ell}$ for a fixed $\BB$. Using the Chernoff bound, the claim is true for all $D_{i,j, \ell}$'s with probability 0.9 if we draw more than $\Theta (\log (n^2 \cdot \vv) t/\epsilon')$ samples. 

Third, we show if $M_{i,j}$ contains enough samples, then with high probability we can distinguish whether  $\|p_{|D_{i,j, \ell}} - \UU_{|D_{i,j, \ell}} \|_2^2$ is at most $\epsilon'^2/{|D_{i,j, \ell}|}$, or it is at least $2\epsilon'^2/{|D_{i,j, \ell}|}$: If we draw $\Theta( t/\epsilon' \cdot  (\log (n^2 \cdot \vv)  \cdot \sqrt{n/t}/\epsilon^2))$ samples, we receive $O((\log (n^2 \cdot \vv)  \cdot\sqrt{n/t}/\epsilon^2) = O((\log (n^2 \cdot \vv)  \cdot\sqrt{{|D_{i,j, \ell}|}}/\epsilon^2)$  samples from any set $D_{i,j, \ell}$ with $p(D_{i,j, \ell}) \geq \epsilon'/t$. Based on~\cite[Theorem 1]{DiakonikolasGPP16}, with probability $1 - 1/3$, we can distinguish whether $\|p_{|D_{i,j, \ell}} - U_{|D_{i,j, \ell}|}\|_2^2 $ is at most $2\epsilon'^2/|D_{i,j, \ell}|$ or at least $\epsilon^2/|D_{i,j, \ell}|$ using $\Theta(\sqrt {|D_{i,j, \ell }|} /\epsilon^2)$ samples. By repeating this $\Theta(\log (n^2 \cdot \vv))$ times and taking the majority answer, we can be assured to obtain the correct answer for the test on all the $D_{i,j, \ell}$'s with probability at least 0.9. Thus, we need $O(\sqrt{n \cdot t}\cdot (\log n + \log \vv)/\epsilon^3)$ samples for this step.

In the above three steps, we need the following number of samples:
\begin{align*}  
    O(t \cdot \log n/\epsilon^2 & + \sqrt{n \cdot t}\cdot (\log n + \log \vv)/\epsilon^3 + \log (n^2 \cdot \vv) t/\epsilon')
    \\ & = O(\sqrt{n\cdot k}/\epsilon^{3.5}) \cdot \text{Polylog}(n, \epsilon^{-1})
    =  \widetilde{O}(\sqrt{n\cdot k}/\epsilon^{3.5})
\,.
\end{align*}
By a union bound, the probability than any of the above steps goes wrong is at most 0.3. Hence, for the rest of the proof, we assume that the algorithm carries out the steps as expected  with probability at least 2/3. Given this assumption, we show in both the completeness case and the soundness case, the algorithm outputs the correct answer.

\noindent
\textit{Completeness:} In this case, there exist a $k$-flat distribution over $\II$, $r$, and a parameter $\alpha^*$ such that $p = (1-\alpha^*) q + \alpha^* r$. 
First, note that $p$ in each $D_{i,j, \ell } \in \widetilde{\DD}$ is close to the uniform distribution. In particular, we have the following lemma.

\begin{restatable}{lemma}{lemSmoothMix} \label{lem:smooth_Mix}
Suppose $p$ is a mixture of $q$ and $r$ with parameter $\alpha$. Let $\II^*$, $\BB$, and $\widetilde{\DD}$, be the partitions we defined earlier. For any non-empty set, $D_{i,j, \ell} \in \DD$, if $j > 1$, then the restriction of $p$ to the set, $p_{|D_{i,j, \ell}}$, is $\epsilon$-close to the uniform distribution in $\ell_1$-distance and $\epsilon/\sqrt{|D_{i,j, \ell}|}$-close to the uniform distribution in $\ell_2$-distance.
\end{restatable}

The lemma implies that for all the $D_{i,j, \ell}$, $p_{|D_{i,j, \ell}}$ is close to the uniform distribution. Hence, the algorithm while considering the segmentation $\II^*$, will not continue with another segmentation since $p_{|D_{i,j, \ell}}$ is being far from uniform, and the algorithm will move on to the next step.

Also,  we show that a $k$-flat function, $f$, exists, because $r$ is a solution itself. We have $\Omega(t /\epsilon'^2)$ samples which is enough to learn the coarsening of $p$ over $\widetilde{\DD}$. Thus, the coarsening of the empirical distribution, $\hat p$, is $\epsilon'$-close to the coarsening of $p$ over $\widetilde{\DD}$. There exists an iteration in the algorithm in which we try a parameter $\alpha$ such that $\alpha - \alpha^*$ is at most $\epsilon'/2$. Therefore, $r$ itself is a solution the algorithm is looking for:
\begin{align*}
    \|\hat p_{\tuple{\widetilde{\DD}}} - & ((1-\alpha)q_{\tuple{\widetilde{\DD}}} + \alpha \,r_{\tuple{\widetilde{\DD}}})\| \leq  \|\hat p_{\tuple{\widetilde{\DD}}} - p_{\tuple{\widetilde{\DD}}} \|_1 
    \\& + \|p_{\tuple{\widetilde{\DD}}} - ((1-\alpha^*)q_{\tuple{\widetilde{\DD}}} + \alpha^*\, r_{\tuple{\widetilde{\DD}}})\| 
    \\ & + \|((1-\alpha^*)q_{\tuple{\widetilde{\DD}}} + \alpha^*\, r_{\tuple{\widetilde{\DD}}})- ((1-\alpha)q_{\tuple{\widetilde{\DD}}} + \alpha\, r_{\tuple{\widetilde{\DD}}})\| 
    \\ & \leq \epsilon' + 0 + \frac{\epsilon'}{2} \cdot \|q_{\tuple{\widetilde{\DD}}} - r_{\tuple{\widetilde{\DD}}}\|_1 \leq 2 \epsilon'.
\end{align*}
Hence, the algorithm will not output \reject. 

\noindent
\textit{Soundness:} In this case, $p$ is $\epsilon$-far from any mixture distribution $q_\alpha = ((1-\alpha)q_{\tuple{\widetilde{\DD}}} + \alpha \,r_{\tuple{\widetilde{\DD}}})$ for any $k$-flat distribution $r$ and $\alpha \in [0,1]$. We have the following structural lemma (similar to Lemma 6 in \cite{BatuFFKRW}) which bounds the distance between $p$ and $q_\alpha$ from above: 
\begin{restatable}{lemma}{lemDistLOneToCoarsening}\label{lem:distL1ToCoarsening}
Assume $p$ and $q$ are two distributions on $[n]$, and let $\widetilde{\DD}$ be a refined division of the domain elements. Then, we have 
$$
\left\|p - q \right\|_1  \leq \left\| p_{\tuple{\widetilde{\DD}}} - q_{\tuple{\widetilde{\DD}}} \right\|_1 + \sum\limits_{D \in \widetilde{\DD}} \left\| p_{|D} - q_{|D}\right\|_1 \cdot \min( p(D), q(D)).$$
\end{restatable}

Since the distance between $p$ and $q_\alpha$ is at least $\epsilon$, we can apply this lemma to obtain a lower bound for the two quantities in the right hand side of the equation above. 
\begin{equation}\label{eq:twoTerm}
\begin{split}
    \epsilon & = 14 \epsilon' < \|p - q_\alpha\|_1 
    \\ & \leq  \left\| p_{\tuple{\widetilde{\DD}}} - (q_\alpha)_{\tuple{\widetilde{\DD}}} \right\|_1 + \sum\limits_{D_{i,j, \ell}} \left\| p_{|D_{i,j, \ell}} - (q_\alpha)_{|D_{i,j, \ell}} \right\|_1  \cdot \min\left(p(D_{i,j, \ell}), q_\alpha(D_{i,j, \ell})\right).
\end{split}
\end{equation}

At least one of the two terms on the right hand side above is greater than  $7\epsilon'$. Net, we show if the algorithm reaches to the point that forms the empirical distribution, then the second term is at most $7\epsilon'$. On the other hand, if the algorithm outputs \accept, then the first term is at most $5\epsilon'$. Hence, these two events cannot happen at the same time while $|p - q|\geq \epsilon$.

Formally, if there is no $D_{i,j, \ell}$ such that causes the algorithm to move forward to the next segmentation, then for each $D_{i,j, \ell}$ either the weight of the set is not larger than $\epsilon'/|\widetilde{\DD}|$, or the $\ell_2$-distance between $P_{|D_{i,j, \ell}}$ and the uniform distribution is not more than $\sqrt {2 \epsilon'^2/|D_{i,j, \ell}|}$. In the following lemma, we show that this situation implies that the second term in Equation~\ref{eq:twoTerm} is at most $6.42\,\epsilon'$.
\begin{restatable}{lemma}{lemUniftoQa}\label{lem:UniftoQa}
Suppose for every non-empty $D_{i,j, \ell}$ in the division $\widetilde{\DD} = \widetilde{\DD}(\II, \BB)$, either $p(D_{i,j, \ell})$ is at most $\epsilon'/|\widetilde{\DD}|$, or  $\|p_{|D_{i,j, \ell}} - \UU_{|D_{i,j, \ell}}\|_2^2$ is at most $2 \epsilon'^2 / |D_{i,j, \ell}|$. Let $q_\alpha$ be a mixture of $q$ and a $k$-flat distribution over $\II$ with an arbitrary $\alpha$ in $[0,1]$. Then, the following holds
$$
\sum_{i=1}^k \sum_{j=1}^\vv \|p_{|D_{i,j, \ell}} - (q_\alpha)_{|D_{i,j, \ell}}\|_1 \cdot \min \left(p(D_{i,j, \ell}), q_\alpha(D_{i,j, \ell}) \right)\leq 6.42 \, \epsilon'\,.
$$
\end{restatable}

On the other hand, if the algorithm outputs \accept, it implies that there exists a function $f$ and $\alpha$, such that $\|\hat p - ((1-\alpha) \, q + \alpha \, f)\|_1$ is at most $2 \epsilon'$. In the following lemma, we show it implies that there exists a $q_\alpha$ such that  $\|p_{\tuple{\DD}} - (q_\alpha)_{\tuple{\DD}}\|_1$ is at most  $5 \epsilon$. 

\begin{restatable}{lemma}{lemkFlatFuncTodist}\label{lem:kFlatFuncTodist}
Assume $p$, $\hat p$, and $q$ are three distributions over $[n]$, and $f:[n] \rightarrow R^+$ is a $k$-flat function over $k$-segmentation $\II$. For a division $\DD$,
suppose $\hat p_{\tuple{\DD}}$ is $\epsilon'$-close to $p_{\tuple{\DD}}$, and there exists $\alpha \in [0,1]$  such that $\|\hat p - ((1-\alpha) \, q + \alpha \, f)\|_1$ is at most $2 \epsilon'$. Then, there exists a $k$-flat distribution $r$, such that $p$ is $5 \epsilon'$-close to the mixture of $r$ and $q$ with parameter $\alpha$.
\end{restatable}

Moreover, outputting \accept means that the two terms in  Equation~\ref{eq:twoTerm} are at most $6.42\epsilon' + 5 \epsilon' < 14 \epsilon'$, which contradicts the fact that one of them has to be $7\epsilon'$. Hence, the proof is complete. 
\end{proof}

\paragraph{A faster algorithm.} In the interest
of a simpler exposition, the algorithm described above
tries all possible $k$-segmentations.  However,
there are at most $O(n^2 \cdot \vv)$ possible subsets that 
could appear as $D_{i,j, \ell}$'s. Hence, one can test 
uniformity of $p$ on each of them separately regardless of $\II$. 
Moreover, finding a $k$-flat function $f$ for which the  $\ell_1$-distance between $\hat p$ and the mixture of $q$ and $f$ is minimized,
can be done via dynamic programming:
we define $d[i,j]$ to be the smallest $\ell_1$-distance between $\hat p$ and mixture of $q$ and any $j$-flat distribution when we consider only the first $i$ elements of the domain. We compute $d[i,j]$  
using the previously computed $d[i',j-1]$:
$$
d[i,j] = \min\limits_{1 \leq i' < i} d[i', j-1] + \mbox{cost}([i',i]),
$$
where the cost$([i',i])$ is defined as follows:
 We set the cost of an interval to infinity if any 
subset of $[i',i]$ which would have appeared in the divisions (i.e, all subsets in such form $[i',i] \cap B_z)$ for $z = 1, \ldots, \vv$) 
does not pass the uniformity test. Otherwise, 
cost$([i,i'])$ is the minimum $\ell_1$-distance between $\hat p$ and a mixture of $q$ and a constant function for the elements in $[i',i]$. 
Since 
we are only looking for 
$k$-flat functions rather than  distributions, 
the updates can be computed locally and independently 
of the rest of segments. 

\subsection{Proofs for Section~\ref{sec:kFlatAlg}}\label{sec:kFlatProofs}

In this section, we provide the proof for the lemmas stated earlier in Section~\ref{sec:kFlatAlg}.

\lemMultiEmpirical* 

\begin{proof}
\label{prf:MultiEmpirical} Suppose we draw $s$ samples from $p$. Let $n_x$ indicate the number of occurrences of $x$ among the samples. Let $\hat p$ be the empirical distribution formed by $s$ samples which means that $\hat p(x) \coloneqq n_x/s$.
The goal is to show that for every segmentation $\II$, the coarsening of $\hat p$ and $p$ over $\widetilde{\DD}(\II,\BB, n, t)$ are $\epsilon'$-close with probability at least $1-\delta$. 
We build on the standard idea that is used to show that $O(n/\epsilon^2)$ samples is sufficient to learn a distribution over $[n]$ within $\epsilon$ error in $\ell_1$-distance.
Consider $\widetilde {\DD}$ which contains  $\Theta(t)$ disjoint subsets of $[n]$. The $\ell_1$-distance between the coarsening of $p$ and the empirical distribution is defined as follows:
\begin{align*}
    \|p_{\tuple{\DD}} - \hat p_{\tuple{\DD}}\|_1 & = \sum_{i=1}^k \sum_{j=1}^\vv \sum_{\ell=1}^{D_{i,j, \ell}} \left|\sum_{x \in D_{i,j, \ell}} p(x) - \sum_{x \in D_{i,j, \ell}} \hat p(x)\right|
    \\& = \sum_{x \in [n]} \mbox{sign}\left( \sum_{x \in D_{i,j, \ell}} p(x) - \sum_{x \in D_{i,j, \ell}} \hat p(x) \right) \cdot (p(x) - \hat p (x)).
\end{align*}
We need to show that the above quantity is at most $\epsilon'$ for any segmentation $\II$ and its corresponding division $\DD(\II,\BB)$. However, we prove a stronger claim:
Suppose we have a collection $C$ of vectors of length $n$ with entries in $\{+1, -1\}$  for which the following is true:
\begin{itemize}
    \item For every refined division $\widetilde\DD$, an every set $D_{i,j, \ell} \in \widetilde{\DD}$, there exists a vector $c \in C$ such that if $x$ is in $ D_{i,j,\ell}$, then $c_x = \mbox{sign}\left( \sum_{x \in D_{i,j, \ell}} p(x) - \sum_{x \in D_{i,j, \ell}} \hat p(x) \right)$.
    
    \item For all $c \in C$, $\sum_{x\in[n]} c_x \cdot (p(x) - \hat p (x))$ is at most $\epsilon'$ with probability at least $1-\delta$.
\end{itemize}

The proof is complete if we establish this claim, so now we focus on proving that the collection $C$ exists. We first put a vector $c$ corresponding to each refined division. Then we show there is an upper bound for the size of the collection. Next, we show since there are not too many vectors in the collection, with high probability, $\sum_{x\in[n]} c_x \cdot (p(x) - \hat p (x))$ is at most $\epsilon'$ for any $c \in C$.

Clearly, there are no more than $2^n$ possible vectors. However, we get a better bound for the cases when $k$ is not arbitrarily large. We begin by
considering a refined division $\widetilde{\DD}$. Fix a set $B \in \BB$. If two elements in $B$ are in the same interval $I_i$ for $i \in [k]$, 
then they will have the same $c_x$ as well. 
Thus, if we sort elements in $B$, and then write the corresponding $c_x$'s,
then we get a sequence of $+1$ and $-1$ where the sign is changed 
in at most $\Theta(t)$ places. 
To uniquely represent the sequence, 
one can determine the indices where the sign 
changed and indicate whether the sequence starts with $+1$ or $-1$. Thus, the total number of such sequences is: 
$$
\sum_{i=0}^{\Theta(t)} 2 \cdot 
\binom{|B|-1}{i} \leq 2 \cdot ((|B|)^{\Theta(t)} + 1) \leq n^{\Theta(t)} \,.
$$
Note that we have at most $\vv$ such subsets of the domain $B \in \BB$. Thus, the total number of vectors in $C$ is at most $\min(2^n, n^{\Theta(t\cdot \vv)})$. 

Next, we show that if we draw enough samples, the probability of  $\sum_{x\in[n]} c_x \cdot (p(x) - \hat p (x)) \geq \epsilon'$ for any $c \in C$ is at most $\delta$.
Fix a vector $c = (c_1, \ldots, c_n)$ in $\{+1, -1\}^n$. Consider the following random process: we draw a sample from $p$, namely $x$; if $c_x$ is one,  output one and otherwise output zero. In other words upon receiving sample $x$, we output $(1 + c_x)/2$. Assume $x_1, \ldots, x_s$ are $s$ samples from $p$ that form the empirical distribution. Suppose that we generate $ b_1, \ldots, b_s$ according to the process using these samples from $p$, i.e., $b_j = (1 + c_{x_j})/2$.  The $b_j$'s are $s$ independent random variables with the following expected value. 
\begin{align*}
    \E{b_j} = \sum_{x = 1}^n p(x) \cdot \frac{1 + c_x}{2} = \frac{1 + \sum_{x} c_x \cdot p(x)}{2}.
\end{align*}
Clearly, the average of $b_j$'s are close to its expectation with high probability, we use this fact to show that $\sum_{x} c_x \cdot (p(x) - \hat p(x))$ are close to zero as well. Recall $n_x$ is the number of occurrences of element $x$ in the sample set. Using the Hoeffding bounds, we achieve:
\begin{equation} \label{eq:prOfDeviation}
\begin{split}
    \Pr{\left|\sum_x c_x  \cdot (p(x) - \hat p(x))\right| \geq \epsilon'}  &= \Pr{\left|\sum_x c_x \cdot p(x) - \sum_x c_x \cdot \frac{n_x}{s}\right| \geq \epsilon'}
    \\ & = \Pr{\left|\sum_x c_x \cdot p(x) - \frac{\sum_{j=1}^s c_{x_j} }{s}\right| \geq \epsilon'}
    \\ & = \Pr{\left|\frac{ 1+ \sum_x c_x \cdot p(x)}{2} - \frac{ s + \sum_{j=1}^s c_{x_j} }{2 \, s}\right| \geq \frac{\epsilon'} 2}
    \\ & =
    \Pr{\left| \frac{1 + \sum_{x} c_x \cdot p(x)}{2} - \frac{\sum_{j=1}^s b_j}{s}\right| \geq \frac {\epsilon'} 2}
    \\ & = \Pr{\left|\E{b_i} - \frac{\sum_{j=1}^s b_j}{s} \right| \geq \frac {\epsilon'} 2} \leq 2 \, \exp\left(-\,s\,{\epsilon'}^2/2\right).
\end{split}
\end{equation}

Therefore, by setting $s = \Theta(\min(n, t \cdot \vv \cdot \log n)\cdot (\log \delta^{-1})/\epsilon'^2)$ and using Equation \ref{eq:prOfDeviation} and a union bound, for every $c \in C$,   $\sum_{x\in[n]} c_x \cdot (p(x) - \hat p (x)) \geq \epsilon'$ is at most $\epsilon'$ with probability $1-\delta$.  This completes the proof.
\end{proof}

\lemSmoothMix*
\begin{proof}
Fix 
a non-empty set $D_{i,j, \ell}$ in $\widetilde{\DD}$ for some $j > 1$. To prove the lemma, we show that the ratio of the maximum and the minimum probability according to $p$ in $D_{i,j, \ell}$ is at most $1+ \epsilon'$.
Consider two elements in $D_{i,j, \ell}$ namely $x$ and $y$ (if there is only one element in $D_{i,j, \ell}$ the claim is apparent). Without loss of generality assume $q(x) \leq q(y)$. By definition of $D_{i,j, \ell}$, $x$ and $y$ are in the same interval of $\II$, so $r(x)$ and $r(y)$ are equal. Thus, we have:
\begin{align}
   1 \leq \frac{p(y)}{p(x)} = \frac{(1-\alpha) q(y) + \alpha\, r(y)}{(1-\alpha) q(x) + \alpha\,r(x)} \leq \frac{q(y)}{q(x)} \leq 1 + \epsilon \,,
\end{align}
the second to last inequality is true, because we have $q(y) \geq q(x) > 0$. Also, the last inequality is true since both $x$ and $y$ are in $B_j$.  In the proof of Lemma 8 in \cite{BatuFFKRW}, it is show that if the ratio of the probabilities in a set, in our case $D_{i,j, \ell}$, is bounded by $(1+\epsilon)$, then for all $x \in D_{i,j, \ell}$, $\left|p(x) - (1/|D_{i,j, \ell}|)\right|$ is at most $\epsilon/|D_{i,j, \ell}|$.  This completes the proof.
\end{proof}

\lemDistLOneToCoarsening*

\begin{proof}
Fix a set in $\widetilde{\DD}$, namely $D$, which $p(D)$ and $q(D)$ are non-zero, we have the following:
\begin{align*}
\left\| p_{|D} - q_{|D}\right\|_1 & = \sum\limits_{x \in D} \left| p_{|D} (x)- q_{|D}(x) \right|  = \sum\limits_{x \in D} \left| \frac {p(x)}{p(D)}- \frac{q(x)}{q(D)} \right| 
\\ & = \sum\limits_{x \in D} \left| \frac {p(x)}{p(D)}- \frac {q(x)}{p(D)}  + \frac {q(x)}{p(D)} - \frac{q(x)}{q(D)} \right| 
\vspace{2mm} 
\\& = \sum\limits_{x \in D} \left| \frac {p(x) - q(x)}{p(D)}-  q(x) \cdot \frac{q(D) - p(D)}{p(D) \,q(D)} \right| 
\\& \geq \frac 1 {p(D)} \sum\limits_{x \in D} \left |p(x) - q(x) \right| - 
\frac {\left| p(D) - q(D) \right|} {p(D)} \cdot  \sum\limits_{x \in D} \frac{q(x)}{q(D)}
\vspace{2mm} \\
& = \frac 1 {p(D)} 
\left( \sum\limits_{x \in D} \left |p(x) - q(x) \right| - 
\left| p(D) - q(D) \right| \right) \,.
\end{align*}

Therefore, we have:
\begin{equation} \label{eq:l_1ToCoarsening_p}
\sum\limits_{x \in D} \left |p(x) - q(x) \right| \leq \left| p(D) - q(D) \right| + \left\| p_{|D} - q_{|D}\right\|_1 \cdot p(D) \, . 
\end{equation}
If we swap $p$ and $q$ in the above inequality, and replicate the equations, we have:
\begin{equation} \label{eq:l_1ToCoarsening_q}
\sum\limits_{x \in D} \left |p(x) - q(x) \right| \leq \left| p(D) - q(D) \right| + \left\| p_{|D} - q_{|D}\right\|_1 \cdot q(D) \, . 
\end{equation}
Putting Equation~\ref{eq:l_1ToCoarsening_p} and Equation~\ref{eq:l_1ToCoarsening_q} together, we get: 
\begin{equation*}
\sum\limits_{x \in D} \left |p(x) - q(x) \right| \leq \left| p(D) - q(D) \right| + \left\| p_{|D} - q_{|D}\right\|_1 \cdot \min(p(D), q(D)) \, . 
\end{equation*}
If at least one of $p(D)$ and $q(D)$ is zero, it implies:
\begin{equation*} 
\sum\limits_{x \in D}  \left |p(x) - q(x) \right| = \left| p(D) - q(D) \right|.
\end{equation*}
Hence, we have:
$$
\begin{array}{ll}
\left\|p - q \right\|_1 &  \leq \sum\limits_{D \in \widetilde{\DD}} \left| p(D) - q(D) \right| + \sum\limits_{D \in \widetilde{\DD}} \left\| p_{|D} - q_{|D}\right\|_1 \cdot \min( p(D), q(D))
\vspace{2mm} \\
& \leq \left\| p_{\tuple{\widetilde{\DD}}} - q_{\tuple{\widetilde{\DD}}} \right\|_1 + \sum\limits_{D \in \widetilde{\DD}} \left\| p_{|D} - q_{|D}\right\|_1 \cdot \min( p(D), q(D))
\, .
\end{array}
$$
\end{proof}

\lemUniftoQa*

\begin{proof}
We first consider a non-empty $D_{i,j, \ell}$ when $j = 1$. Since $j = 1$, $D_{i,j, \ell}$ is a subset of $B_1$. For each $x \in D_{i,j, \ell}$, $q(x)$ is at most $\epsilon'^2/n$. Also, $r$ is a $k$-flat on $\II$, and since $D_{i,j, \ell}$ is a subset of $I_i$, for all $x \in D_{i,j, \ell}$, $r(x)$ is the same. We denote this quantity, $r(x)$, by $b$. Here, we prove that either $q_{\alpha}(D_{i,j, \ell})$ is small, or $q_{\alpha}(D_{i,j, \ell})$ has to be close to uniform. 

We have two cases. First, suppose $\alpha \cdot b$ is at most $\epsilon'/n$. In this case, $q_\alpha(D_{i,j, \ell})$ is at most $\epsilon'^2 |D_{i,j, \ell}|/n$. Thus, the total weight of such sets, sum of $q_\alpha(D_{i,j, \ell})$'s, is at most $\epsilon'^2$. Second, assume $\alpha \cdot b$ is greater that $\epsilon'/n$. On the other hand, $q(x)$ is at most $\epsilon'^2/n$. These two facts implies for each $x$ in $D_{i,j, \ell}$:
\begin{align*}
    |D_{i,j, \ell}| \cdot q_\alpha(D_{i,j, \ell}) \left|(q_\alpha)_{|D_{i,j, \ell}}(x) - \frac 1 {|D_{i,j, \ell}|}\right| & = |D_{i,j, \ell}| \cdot q_\alpha(D_{i,j, \ell}) \cdot \left| \frac{q_\alpha(x)}{q_\alpha(D_{i,j, \ell})} - \frac 1 {|D_{i,j, \ell}|} \right| \\ &  = \left| |D_{i,j, \ell}| \cdot (1-\alpha) q(x) - (1-\alpha) \sum_{x \in D_{i,j, \ell}} q(x)  \right|  \\ & \leq |D_{i,j, \ell}| \frac{\epsilon'^2}{n} \leq \epsilon'\, |D_{i,j, \ell}| \,\alpha \,b   \leq \epsilon' \cdot q_\alpha(D_{i,j, \ell}) \,.
\end{align*}

Therefore, the $\ell_2^2$-distance  between $(q_\alpha)_{|D_{i,j, \ell}}$ and the uniform distribution is bounded:
\begin{align*}
    \|(q_\alpha)_{|D_{i,j, \ell}}(x) - \UU_{|D_{i,j, \ell}}\|_2^2 = \sum_{x \in D_{i,j, \ell}} \left((q_\alpha)_{|D_{i,j, \ell}}(x) - \frac 1 {|D_{i,j, \ell}|}\right)^2 \leq  \frac{\epsilon'^2}{|D_{i,j, \ell}|}\,.
\end{align*}

Note that if $j$ is greater than one, the $\ell_2$-distance between 
$(q_\alpha)_{|D_{i,j, \ell}}$ and the uniform distribution is bounded by $\epsilon'/\sqrt{|D_{i,j, \ell}|}$ as well.  
Therefore, if $p_{|D_{i,j, \ell}}$ is close uniform distribution, it has to be close to $(q_\alpha)_{|D_{i,j, \ell}}$ as well. That is,
\begin{align*}
\|p_{|D_{i,j, \ell}} - (q_\alpha)_{|D_{i,j, \ell}}\|_1 & \leq \|p_{|D_{i,j, \ell}} - \UU_{|D_{i,j, \ell}}\|_1 + \|\UU_{|D_{i,j, \ell}} - (q_\alpha)_{|D_{i,j, \ell}}\|_1 
\\& \leq \sqrt {|D_{i,j, \ell}|} \left(\|p_{|D_{i,j, \ell}} - \UU_{|D_{i,j, \ell}}\|_2 + \|\UU_{|D_{i,j, \ell}} - (q_\alpha)_{|D_{i,j, \ell}}\|_2 \right)
\leq 2.42 \epsilon'\,.
\end{align*}

Hence, given the discussion above there are three possibilities for $D_{i,j, \ell}$. (i) $p(D_{i,j, \ell})$ is at most $\epsilon'/(k.\vv)$. Since $\ell_1$-distance is at most $2$, the total contribution of these sets in the sum below is at most $2\epsilon'$. (ii) $q_\alpha(D_{i,j, \ell})$ is at most $\epsilon'^2 |D_{i,j, \ell}|/n$, so the total contribution of these sets is at most $2 \epsilon'^2$. (iii) $\|p_{|D_{i,j, \ell}} - (q_\alpha)_{|D_{i,j, \ell}}\|_1$ is at most $2.42 \, \epsilon'$.
$$
\|p_{|D_{i,j, \ell}} - (q_\alpha)_{|D_{i,j, \ell}}\|_1 \cdot \min\left(p(D_{i,j, \ell}), q_\alpha (D_{i,j, \ell}) \right) \leq 6.42\,\epsilon' \,.
$$
Hence, the proof is complete. 
\end{proof}

\lemkFlatFuncTodist*
\begin{proof}
First, consider a degenerate case. If $\alpha = 0$, then the claim is trivially true by the triangle inequality: 
$ \|p - q\|_1 \leq \|p - \hat p \|_1 + \|\hat p - q \|_1 \leq 3\epsilon'\,.$
Thus, assume $\alpha > 0$.

For now, consider the case that there exists $x$ such that $f(x)$ is not zero, so $\sum_x f(x)$ is greater than zero.  First, we show that since $\hat p$ is close to the mixture of $q$ and $f$, the sum of the $f(x)$'s has to be close to one. That is, 
\begin{equation}\label{eq:1-f_bounded}
\begin{gathered}
    -2\,\epsilon' \leq \sum_x \hat p(x) - ((1-\alpha)\,q(x) - \alpha \, f(x)) \leq 2\,\epsilon' \quad \Rightarrow
\\
    -2\,\epsilon' \leq 1 - (1-\alpha) - \alpha \cdot \sum_x  f(x)) \leq 2\,\epsilon' \quad \Rightarrow
\\
    -\frac{2\,\epsilon'}{\alpha} \leq 1 -  \sum_x  f(x) \leq \frac{2\,\epsilon'}{\alpha} \quad \Rightarrow \quad \left|1 -  \sum_x  f(x)\right| \leq \frac{2\,\epsilon'}{\alpha}\,.
\end{gathered}
\end{equation}
We define $r:[n]\rightarrow [0,1]$ to be the normalization of $f$ for which $r(x) = f(x) / \sum_x f(x)$ for all $x$ in the domain. If $f$ is a $k$-flat function, then $r$ will be a $k$-flat distribution. Now, we show that the mixture of $q$ and $f$ is close to the mixture of $q$ and $r$ with mixture parameter $\alpha$.
\begin{align*}
    \|(1-\alpha) \, q + \alpha \, f &  - (1-\alpha) \, q + \alpha \, r\|_1 = \sum_x \left|(1-\alpha) \, q(x) + \alpha \, f(x)  - (1-\alpha) \, q(x) + \alpha \, r(x)\right|
\\ 
    &  = \alpha \cdot \sum_x \left|f(x) - r(x)\right| = \alpha \cdot \sum_x  f(x) \cdot \left|1 - \frac{1}{\sum_y f(y)}\right|
\\
    &  = \alpha \cdot \left|\sum_x  f(x)  - 1 \right| \leq 2\,\epsilon'\,,
\end{align*}
where the last inequality is due to Equation~\ref{eq:1-f_bounded}. Moreover, by the triangle inequality, we have:
$$\|\hat p - ((1-\alpha) \, q + \alpha \, r)\|_1 \leq 
\|\hat p - ((1-\alpha) \, q + \alpha \, f)\|_1 + \|(1-\alpha) \, q + \alpha \, f  - (1-\alpha) \, q + \alpha \, r\|_1 \leq 4\,\epsilon'\,.
$$

Now, assume $f(x)$ is zero for all $x$ in $[n]$. We show that if we set $r$ to be the uniform distribution over $[n]$, the same result holds. First, observe that the uniform distribution is a $k$-flat distribution for any $k \geq 1$. Then we show that $(1-\alpha)\,q + \alpha r$ is $4\epsilon'$-close to $p'$. Since $r(x) = 1/n$ for all $x$ in $[n]$, 
\begin{align*}
     \|\hat p - ((1-\alpha) \, q + \alpha \, r)\|_1 \leq \|\hat p - ((1-\alpha) \, q )\|_1 + \alpha \,.
\end{align*}
On the other hand, since $\hat p$ is $2\epsilon'$-close to $(1-\alpha)\,q$, one can show that $\alpha$ is at most $\epsilon'$.
\begin{align*}
    2\,\epsilon' \geq \|\hat p - ((1-\alpha) \, q + \alpha \, f)\|_1 \geq \sum_x \hat p(x) - (1-\alpha)\,q(x) = 1 - 1 + \alpha = \alpha\,.
\end{align*}
Therefore, whether $\sum_x f(x)$ is zero or not, there exists a $k$-flat distribution, $r$ for which
$\|\hat p - ((1-\alpha) \, q + \alpha \, r)\|_1$ is at most $4\,\epsilon'$.
Since $p'$ is $\epsilon'$-close to $p$, and by triangle inequality, we have:
$$ \|p - ((1-\alpha) \, q + \alpha \, r)\|_1 \leq 5 \epsilon' \,,
$$
which concludes the proof.
\end{proof}

\section{Lower bounds}
\label{sec:lb}

In this section, we present lower bounds for testing mixtures in different settings discussed earlier. 

\LBBigness
\begin{proof}
We prove by showing a reduction from mixture testing to testing bigness property of distributions. A distribution  called \emph{$T$-big} if the probability of any domain element is at least $T$~\cite{aliakbarpourGPRY18}. In addition, they showed there exist two constant parameters $\epsilon$ and $\beta$ and two family of distributions, namely $\FF^+$ and $\FF^-$, such that the following is true
\begin{itemize}
    \item All distribution in $\FF^+$ are $1/(\beta n)$-big. 
    \item All distribution in $\FF^-$ are $\epsilon$-far from being $1/(\beta n)$-big. Moreover, all the probability of each element according to the distributions is either zero or at least $1/(\beta n)$.
    \item Using $o(n/\log n)$ samples from a distribution in the families, no algorithm can distinguish whether the distribution was from $\FF^+$ or $\FF^-$ with probability at least $2/3$.
\end{itemize}

Let $\epsilon = 1/\beta$.  We show that any algorithm that can test mixtures as described in theorem, can distinguish $\FF^+$ and $\FF^-$ with high probability. 

First, we show that for any $1/(\beta n)$-big distribution, denoted by $p^+$, there exists  distribution $\eta$ such that $p^+$ is a mixture of $\eta$ and uniform distribution, meaning $p^+ =  \alpha \, \eta + (1-\alpha)\, \UU$ for $\alpha = 1/\beta$. Let $\eta$ assign the following probability to the $i$th element of the domain:
$$
\eta(i) = \frac{p(i) - \frac 1{\beta n}}{1 - \frac 1 \beta} \,.
$$
It is not hard to see that $\eta$ as defined above is a probability distribution. Since $p$ is $1/(\beta n)$-big distribution, all the $p(i)$ are at least $1/(\beta n)$, so all the $\eta(i)$'s are non-negative. Also, $\sum_i \eta(i)= 1$.  Clearly, $p^+$ is a mixture in the form $\alpha \, \eta + (1-\alpha)\, \UU$ for $\alpha = 1/\beta = \epsilon$.

Note that for any distribution $p^-$ in $\FF^-$ there is at least one element (in fact many elements) that has probability zero. Otherwise, all elements would have probability at least $1/(\beta n)$, and the distribution would be big. On the other hand, any distribution that is mixed with uniform with parameter $\alpha < 1$ cannot have any zero probability element. Thus, $p^-$ is not a mixture of the form $\alpha \, \eta + (1-\alpha)\, \UU$ when $\alpha \neq 1$.

Thus, any algorithm that can test mixture property as defined in the theorem has to accept $p^+$ and reject $p^-$. However, we know this is not possible unless the algorithms gets $\Omega(n/\log n)$ samples. This completes the proof. 
\end{proof}



\begin{proposition} \label{prop:lb_unknown_q_U}
When we have sample access to $q$ and $p$, any closeness tester in the presence of uniform noise $\Omega\left(\max\left( n^{2/3}/\epsilon^{4/3}, \sqrt n/ \epsilon^2\right)\right)$ samples.
\end{proposition}
\begin{proof}\; 
First, note that one can reduce testing uniformity to this problem by setting $q$ equal to the uniform distribution.  Therefore, it requires at least $\Omega(\sqrt n / \epsilon^2)$ samples by the lower bound for uniformity testing shown in  \cite{Paninski:08}. 

Now, we establish that  $\Omega(n^{2/3}/\epsilon^{4/3})$ many samples is also required. Without loss of generality, assume $\epsilon \geq 4^{3/4} / n^{1/4}$. Otherwise $\sqrt n / \epsilon^2$ would be the dominating term in the lower bound up to a constant factor. To prove the lower bound, we use two distributions (and any random relabeling of them) used in proving lower bounds for testing closeness of distributions  \cite{BFRSW, VV11, ValiantV14, ChanDVV14}. 
More precisely, we define two distributions $p^*$ and $q^*$ such that distinguishing $(p^*, q^*)$ and $(q^*, q^*)$ (and any random relabeling of them) requires $\Omega(n^{2/3}/\epsilon^{4/3})$ samples. On the other hand, we show that any $\Omega (q^*, \UU, \epsilon)$-mixture tester has to distinguish $(p^*, q^*)$ and $(q^*, q^*)$. Thus, the statement of the proposition is concluded. 

Let $a = 4\epsilon/n$ and $b = \epsilon^{4/3} / n ^{2/3} $. Consider three disjoint subset of domain elements $[n]$, namely $A$, $B$, and $C$ each of size $(1-\epsilon)/b$, $\epsilon/a$, and $\epsilon/a$ respectively. Let $p$ and $q$ be the following distributions:
$$p^* = b \ind_A + a \ind_B, \quad \quad q^* = b \ind_A + a \ind_C.$$
Note that $p^*$ is $\epsilon$-far from any mixture distribution of $q^*$ and $\UU$ with parameter $\alpha \in [0, 1]$, since
\begin{align*}
\|p^* -  q^*_\alpha\|_1 & = \sum\limits_{i} |p^*(i) - (1-\alpha)q^*(i) - \alpha/n | 
\\
& \geq \sum\limits_{i \in B} |a - \alpha/n |   + \sum\limits_{i \in C} |-(1 - \alpha)a-\alpha/n|
\\
& \geq \frac{n}{4} \left( |a - \alpha/n | + |(1 - \alpha)a + \alpha/n| \right) 
\\
& \geq  \frac{n}{4} \cdot a \geq \epsilon.
\end{align*}
Clearly, in the case where $p = q^*$ and $q = q^*$, $p$ is a mixture of $q$ and $\UU$ with mixture parameter $\alpha = 0$, and in the case where $p = p^*$ and $q = q^*$, $p$ is $\epsilon$-far from any mixture distribution of $q^*$ and $\UU$. Thus, a $(q, \UU, \epsilon)$-mixture tester has to distinguish between $(q^*, q^*)$ and $(p^*, q^*)$. By proposition 4.1 in \cite{ChanDVV14}, we know that this task requires $\Omega(n^{2/3}/\epsilon^{4/3})$ samples.
\end{proof}

\bibliographystyle{alpha}
\bibliography{main.bib}

\end{document}